\documentclass[12pt,nowarn, openany, electronicsig]{pikeuncdiss}

\usepackage{latexsym}
\usepackage{amssymb}
\usepackage{amsbsy}
\usepackage[all]{xy}
\usepackage{enumerate}
\usepackage{amscd}
\usepackage{amsmath,graphicx,amsthm,mathrsfs,verbatim}
\usepackage{ifthen}
\usepackage{graphicx}
\usepackage{ytableau}
\usepackage{eucal}
\usepackage{setspace}

\usepackage[pdftex,bookmarks,colorlinks=true,citecolor=black,filecolor=black,linkcolor=black,urlcolor=black,bookmarksdepth,pdftitle={Explicit formulas for local Mellin transforms}, pdfauthor={Adam Graham-Squire}, pdfsubject={Mathematics}, pdfkeywords={matrix singularities, free divisors, complex hypersurfaces,singular Milnor numbers}]{hyperref}

\theoremstyle{plain}

\newtheorem*{Thm*}{Theorem}

\newtheorem*{Corollary*}{Corollary}

\newtheorem*{Proposition*}{Proposition}

\theoremstyle{plain}

\newtheorem{corollary*}[theorem]{Corollary}

\theoremstyle{definition}

 \newtheorem*{ex}{Example}
 \newtheorem{thm}{Theorem}[section]
 \newtheorem{prop}[thm]{Proposition}
 \newtheorem{lem}[thm]{Lemma}
 \newtheorem{cor}[thm]{Corollary}

\theoremstyle{remark}

\numberwithin{equation}{chapter}
\numberwithin{figure}{chapter}

\begin{document}
\frontmatter

\title{New Descriptions of Demazure Tableaux and Right Keys, with Applications to Convexity}
\author{Matthew J. Willis}
\advisor{Robert Proctor}
\readers{Prakash Belkale}{Patrick Eberlein}{Shrawan Kumar}{Richard Rimanyi}
\thesisabstract{The right key of a semistandard Young tableau is a tool used to find Demazure characters for $sl_n(\mathbb{C})$.  This thesis gives methods to obtain the right and left keys by inspection of the semistandard Young tableau.  Given a partition $\lambda$ and a Weyl group element $w$, there is a semistandard Young tableau $Y_\lambda(w)$ of shape $\lambda$ that corresponds to $w$.  The Demazure character for $\lambda$ and $w$ is known to be the sum of the weights of all tableaux whose right key is dominated by $Y_\lambda(w)$.  The set of all such tableaux is denoted $\mathcal{D}_\lambda(w)$.  Exploiting the method mentioned above for obtaining right keys, this thesis describes the entry at each location in any $T \in \mathcal{D}_\lambda(w)$.  Lastly, we will consider $\mathcal{D}_\lambda(w)$ as an integral subset of Euclidean space.  The final results present a condition that is both necessary and sufficient for this subset to be convex.}
\thesisdedication{}

\thesispreface[Acknowledgements]
{I would first like to thank the mathematics faculty at the University of North Carolina at Chapel Hill.  In particular, the members of my committee have been quite helpful.  I am especially grateful to my advisor Robert Proctor for his insight, advice, patience and perseverance throughout this entire process.  My academic life has improved by leaps and bounds thanks to having worked with him these past three years.  Also, the exposition of this thesis has benefited greatly from his comments.

I would also like to thank Sarah Mason for her remarks and encouragement, reading [Mas] led to the creation of the method presented in Chapter 2, whose original proof referred to her Corollary 5.1.  Thanks are also due to Vic Reiner for suggesting the more direct proof given in Section 2.5, and for encouraging the exploration of conditions for convexity.  In addition, both Mark Shimozono and Jim Haglund provided valuable perspective to my advisor through conversations.  Also, beneficial insight was provided by Keith Schneider in regards to counting flagged Schur tableaux, which sparked the beginning of my interest in convex polytopes.

I would also like to thank my fellow graduate students and friends here at UNC, many of whom have helped keep me focused on the important things.  I specifically want to thank Keith, Brandyn, Nate and Emily for their mathematical contributions.

Lastly, I want to thank my family and specifically my parents for their continued support and encouragement over the years, without which I would not have survived this process.}

\begingroup
\topskip0pt
\vspace*{1in}%
\endgroup

\thispagestyle{empty}

\centerline{\Large{\textbf{New Descriptions of Demazure Tableaux and Right}}}
\centerline{\Large{\textbf{Keys, with Applications to Convexity}}}

\vspace{1in}

\centerline{\large{Matthew J. Willis}}

\vspace{1in}

\singlespace

\noindent A dissertation submitted to the faculty of the University of North Carolina at Chapel Hill in partial fulfillment of the requirements for the degree of Doctor of Philosophy in the Department of Mathematics.

\vspace{1in}

\centerline{Chapel Hill}
\centerline{2012}

\vspace{1in}

\doublespacing

\hspace{95mm} Approved by,

\hspace{95mm} Robert Proctor

\hspace{95mm} Prakash Belkale

\hspace{95mm} Patrick Eberlein

\hspace{95mm} Shrawan Kumar

\hspace{95mm} Richard Rimanyi

\newpage

\begingroup
\topskip0pt
\vspace*{1in}%
\endgroup

\centerline{\large{\textbf{Abstract}}}

\vspace{1pc}
\centerline{MATTHEW J. WILLIS:  New Descriptions of Demazure Tableaux and Right Keys,}
\centerline{with Applications to Convexity}
\centerline{(Under the direction of Robert Proctor)}

\vspace{1pc}
\doublespacing
The right key of a semistandard Young tableau is a tool used to find Demazure characters for $sl_n(\mathbb{C})$.  This thesis gives methods to obtain the right and left keys by inspection of the semistandard Young tableau.  Given a partition $\lambda$ and a Weyl group element $w$, there is a semistandard Young tableau $Y_\lambda(w)$ of shape $\lambda$ that corresponds to $w$.  The Demazure character for $\lambda$ and $w$ is known to be the sum of the weights of all tableaux whose right key is dominated by $Y_\lambda(w)$.  The set of all such tableaux is denoted $\mathcal{D}_\lambda(w)$.  Exploiting the method mentioned above for obtaining right keys, this thesis describes the entry at each location in any $T \in \mathcal{D}_\lambda(w)$.  Lastly, we will consider $\mathcal{D}_\lambda(w)$ as an integral subset of Euclidean space.  The final results present a condition that is both necessary and sufficient for this subset to be convex. 

\doublespacing

\currentpdfbookmark{Table of Contents}{tableofcontents}
\tableofcontents

\chapter*{Introduction}
\thispagestyle{empty}

This thesis is separated into three chapters.  In Chapter 1 we give a synopsis of the relevant notation, definitions, conventions, and background that we will use.

A Demazure character of a ``semistandard Lie algebra'' is determined by a ``highest weight'' $\lambda$ and a ``Weyl group'' element $w$.  We will work entirely in case $A_{n-1}$.  This thesis is inspired by the result of Lascoux and Sch{\"u}tzenberger concerning Demazure characters in Case $A_{n-1}$ that appears as Theorem 1 in [RS].  Here the Weyl group of the Lie algebra $sl_n(\mathbb{C})$ is precisely the symmetric group $S_n$.  Theorem 1 of [RS] says that one can obtain the Demazure character here by summing the weights of all ``semistandard Young tableaux'' whose ``right keys'' are dominated by ``the key of $w$''.

Chapter 2 introduces a new method for obtaining the right key of a semistandard Young tableau, called ``the scanning method''.  This method is the first that does not require any combinatorial objects or notions other than ordinary semistandard tableaux.  We will also briefly mention how to obtain the ``left key'' of a semistandard Young tableau.  This chapter has been submitted for publication and posted on the arXiv [Wi].

Chapter 3 then studies the set of all tableaux contributing to a particular Demazure character.  Given such a ``Demazure tableau'' $T$, the first result of Chapter 3 describes the possible entries at each location in $T$ based on the application of the scanning method.  As a result, we establish an entry-wise condition to determine whether or not a given semistandard Young tableau contributes to a particular Demazure character.

In the latter part of Chapter 3 we consider a special kind of permutation $w$, characterized by a certain pattern that may arise in the ``one rowed form'' of a given permutation.  Due to time limitations, for the remainder of Chapter 3 we restrict our attention to $n$-partitions $\lambda$ of the form $(d, ... , d, d-1, ... , 2, 1)$ for some $1 \leq d \leq n$.  (The most important of these $\lambda$'s is the ``staircase shape''; here $d = n$ and $\lambda = (n, n-1, ... , 2, 1)$.)  In Proposition 3.7.1, we present a new description of the set of Demazure tableaux for these $\lambda$ and any 312-avoiding $w$.  As noted in Corollary 3.7.2, it follows that if $w$ is 312-avoiding, then the set of Demazure tableaux is a convex polytope in $\mathbb{Z}^{|\lambda|}$.  In Proposition 3.8.1, we show that, when $\lambda = (n, n-1, ... , 2, 1)$, if $w$ is ``312-containing'', then the set of Demazure tableaux is not a convex polytope.  Theorem 3.9.1 summarizes by stating that the set of Demazure tableaux for $\lambda = (n, n-1, ... , 2, 1)$ is a convex polytope if and only if $w$ is 312-avoiding.

In a forthcoming paper, we will extend Proposition 3.7.1 and show that $w$ being 312-avoiding implies that the set of Demazure tableaux is equal to a set of flagged Schur tableaux.  Thus it will follow that the Demazure character for a 312-avoiding $w$ is equal to a flagged Schur polynomial.  Proposition 3.8.1 already implies that if $w$ is 312-containing, then the set of Demazure tableaux is not equal to any set of flagged Schur tableaux.  Hence the Demazure character for a 312-containing $w$ does not equal any flagged Schur Polynomial.

The two papers mentioned next consider general $\lambda$.  In Theorem 25 of [RS], Reiner and Shimozono provided a necessary and sufficient tableau condition on $w$ characterizing when a Demazure character equals a ``flagged Schur polynomial''.  In Theorem 14.1 of [PS], Postnikov and Stanley replaced this sufficient condition with the more familiar ``312-avoiding'' condition for a permutation.  They also note in Chapter 15 of [PS] that the set of tableaux used to obtain the flagged Schur polynomial is a ``convex polytope'' in $\mathbb{Z}^{l(w)}$, as was observed by Kogan [Ko].  Our results are stronger than those in [RS] and [PS] in that our results are concerned with actual sets of Demazure tableaux, and are not just statements at the level of characters.

%
%
\mainmatter
\chapter{Background}

\section{Introduction}

In this section we will cover the most widespread notation, conventions, and definitions used in this
thesis.  Among the most important definitions are those for ``semistandard Young tableaux'' (Section 1.3), ``Demazure characters''
(Section 1.14), and ``right keys'' (Section 1.17).  In Section 1.3, we provide a tableau description of the well-known ``Schur function''.  In Section 1.18, we give an analogous tableau description for Demazure characters, which appears as Theorem 1 in [RS].  We will also state the ``Demazure Character Formula'' in Section 1.15, and the computational implementation of it from [RS] in Section 1.16.

\section{Partitions and Diagrams}

Throughout this thesis, there is always some fixed $n \geq 1$.  Set $[n] = \{ 1, 2, ... , n-1, n \}$.  Given $p, q \in [n]$, define $[p, q] := \{ p, p+1, ... , q \}$.  Define an \textit{n-partition} to be a sequence $\beta = (\beta_1 , ... , \beta_n)$ of nonnegative integers such that $\beta_1
\geq ... \geq \beta_n \geq 0$.

The \textit{Young diagram of $\beta$} is the diagram with $\beta_i$ left justified boxes in row $i$ for $1 \leq i \leq n$.  To emphasize the importance of columns over rows, the box at the intersection of the $j$th column with the $i$th row is denoted by $(j, i) \in \lambda$.  Note that the number of boxes in the Young diagram of $\beta$ is $|\beta| := \beta_1 + ... + \beta_n =: N$.  Given an $n$-partition $\beta$, define $\zeta := \zeta_\beta := (\zeta_1, ... , \zeta_{\beta_1})$ to be the sequence of column lengths of the Young diagram of $\beta$.

\begin{ex}
Let $n = 5$ and let $\beta = (5,3,2,2,1)$.  Then the Young diagram of $\beta$ is

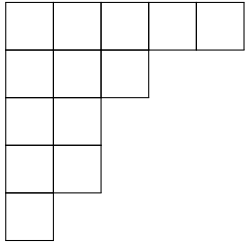
\centerline{$\ydiagram{5,3,2,2,1}$}

\vspace{1pc}\noindent  Here $\zeta = \zeta_\beta = (5, 4, 2, 1, 1)$.

\end{ex}

Given an $n$-partition $\beta$, consider the sequence of nonnegative integers $\gamma = (\gamma_i)_{i=1}^n$ where $\gamma_i = \beta_i - \beta_{i+1}$ for $1 \leq i \leq n-1$ and $\gamma_n = \beta_n$.  Note that $\beta$ can be recovered from $\gamma$ by setting $\beta_i = \gamma_i + ... + \gamma_n$ for $1 \leq i \leq n$.

\section{Semistandard Tableaux and Schur Functions}

Fix an $n$-partition $\beta$.  An \textit{n - semistandard Young tableau of shape $\beta$} is a filling of the Young diagram of $\beta$ with entries from $[n]$ such that entries are weakly increasing across the rows and strictly increasing down the columns.  Henceforth we more simply write semistandard tableau (of shape $\beta$).  The entry in $T$ at the location $(j, i)$ is denoted $T(j, i)$.  By convention, define $T(0, k) := k$ for $1 \leq k \leq \zeta_1$, and $T(j, \zeta_j + 1) := n+1$ for all $1 \leq j \leq \lambda_1$.  Note that a tableau $T$ is semistandard if and only if $T(j,i) \in [T(j-1,i), T(j,i+1) - 1]$ for all $(j,i) \in \lambda$.  Let $\mathcal{T}_{\beta}$ denote the set of all semistandard tableaux of shape $\beta$.  For $T, T' \in \mathcal{T}_\beta$, we write $T \leq T'$ if and only if $T(j,i) \leq T'(j,i)$ for all $(j,i) \in \lambda$;
this partial ordering is known as \textit{entry-wise comparison}.

The ``Schur function'' is a symmetric polynomial specified by a partition and is defined in (for example) [Mac, p.24].  Given $T \in \mathcal{T}_{\beta}$, let $d_j$ be the number of times the entry $j$ appears in $T$.  Let $t_1, ... , t_n$ be indeterminates.  Define the \textit{weight of $T$} to be $wt(T)$ $:= \prod_{j=1}^n t_j^{d_j}$.  Then define $s_{\beta}(t) := \sum_{T \in \mathcal{T}_{\beta}} wt(T)$.  This sum $s_{\beta}(t)$ is equal to the Schur function specified by $\beta$ [St, p. 310].

\begin{ex}
Let $n = 3$ and $\beta = (2,1,0)$.  Here

$$\mathcal{T}_\beta = \{ \hspace{1mm}
\ytableausetup{centertableaux, boxsize = 1.25em}
\begin{ytableau}
1 & 1 \\
2  \\
\end{ytableau}, \hspace{1mm}
\ytableausetup{centertableaux, boxsize = 1.25em}
\begin{ytableau}
1 & 1 \\
3  \\
\end{ytableau}, \hspace{1mm}
\ytableausetup{centertableaux, boxsize = 1.25em}
\begin{ytableau}
1 & 2 \\
2  \\
\end{ytableau}, \hspace{1mm}
\ytableausetup{centertableaux, boxsize = 1.25em}
\begin{ytableau}
1 & 2 \\
3  \\
\end{ytableau}, \hspace{1mm}
\ytableausetup{centertableaux, boxsize = 1.25em}
\begin{ytableau}
1 & 3 \\
2  \\
\end{ytableau}, \hspace{1mm}
\ytableausetup{centertableaux, boxsize = 1.25em}
\begin{ytableau}
1 & 3 \\
3  \\
\end{ytableau}, \hspace{1mm}
\ytableausetup{centertableaux, boxsize = 1.25em}
\begin{ytableau}
2 & 2 \\
3  \\
\end{ytableau}, \hspace{1mm}
\ytableausetup{centertableaux, boxsize = 1.25em}
\begin{ytableau}
2 & 3 \\
3  \\
\end{ytableau} \hspace{1mm}
\},$$

and hence $s_\beta = t_1^2t_2 + t_1^2t_3 + t_1t_2^2 + 2t_1t_2t_3 + t_1t_3^2 + t_2^2t_3 + t_2t_3^2$.

\end{ex}

\section{A Second Schur Function Description}

It is well known that $n$-partitions index the equivalence classes of irreducible polynomial representations of $GL_n(\mathbb{C})$ [St, p. 442].  Fix an $n$-partition $\beta = (\beta_1, ... , \beta_n)$.  Denote a corresponding representation of $GL_n(\mathbb{C})$ by $\phi_\beta$.  Let ch$\phi(t) \in \mathbb{Z}[t_1,...,t_n]$ denote the character of the restriction of a representation $\phi$ to the diagonal matrices in $GL_n(\mathbb{C})$, which have diagonal entries $t_1, t_2, ... , t_n$.  Then one has ch$\phi_\beta(t) = s_\beta(t)$, the Schur function specified by $\beta$ [St, p. 442].

\section{The Lie Algebra and Its Cartan Subalgebra}

The simple Lie algebra $sl_n (\mathbb{C}) =: \mathfrak{g}$ of trace-free matrices is of type $A_{n-1}$.  Let $\mathfrak{h}$ denote the Cartan subalgebra of $\mathfrak{g}$ consisting of trace-free diagonal matrices.  We depict elements of $\mathfrak{h}$ with column vectors from $\mathbb{R}^n$ whose entries sum to zero.  For example, when $n = 2$, the element $\begin{pmatrix}  x & 0 \\  0 & -x \\ \end{pmatrix}$ $\in \mathfrak{h}$ is written as $(x, -x)^T$.

\section{The Root System and Simple Reflections}

Let $\Phi \subseteq \mathfrak{h}_{\mathbb{R}}^*$ denote the set of roots of $\mathfrak{g}$.  Elements of $\mathbb{R}^{n}$ are depicted with column vectors $(x_1, x_2, ... , x_n)^T$ of length $n$.  We embed $\mathfrak{h}_{\mathbb{R}}^*$ into $\mathbb{R}^n$ in such a way that the simple roots are $\alpha_1 := (1,-1,0,...,0)^T, \alpha_2 := (0,1,-1,0...,0)^T, ... , \alpha_{n-1} := (0,...,0,1,-1)^T $.  For example, when $n = 2$ the evaluation of $\alpha_1$ on $(1/2, -1/2)$ is $1/2 - (-1/2) = 1$.  Define $\Delta := \{ \alpha_1, ... , \alpha_n \}$ to be the \textit{base} of $\Phi$.

For $1 \leq i \leq n-1$, define $s_i$ to be the action of reflecting $\mathbb{R}^n \supseteq \mathfrak{h}_{\mathbb{R}}^*$ in
this depiction with respect to the simple root $\alpha_i$.  (The reflecting of $\mathbb{R}^n$ is done with respect to the dot product on $\mathbb{R}^n$.)  It can be seen that $s_i$ interchanges the entries in the $i$th and $(i+1)$st locations of a column vector (and leaves the other locations unchanged).

For $1 \leq i \leq n$, let $\epsilon_i$ be the vector with a 1 in the $i$th position and zeros elsewhere.  Here $\epsilon_i$ is the coordinate function on $\mathfrak{h}_\mathbb{R}$ that extracts the $i$th entry of diagonal matrices.  The set $\{ \epsilon_1, ... , \epsilon_n \}$ is the \textit{axis basis} of $\mathbb{R}^n$.  Also, we have $\alpha_i = \epsilon_i - \epsilon_{i+1}$.  Note that $s_i(\epsilon_i) = \epsilon_{i+1}$, $s_i(\epsilon_{i+1}) = \epsilon_i$, and $s_i(\epsilon_j) = \epsilon_j$ if $j \neq i, i+1$.  Also note that, as linear functions restricted to $\mathfrak{h}$, the relation $\epsilon_1 + ... + \epsilon_n = 0$ is satisfied.

\section{The Weyl Group}

Let $W$ denote the Weyl group of $\mathfrak{g}$:  this is the subgroup of $GL(\mathfrak{h}_{\mathbb{R}}^*)$ generated by the
$s_i$'s for $1 \leq i \leq n-1$.  Let $w_0$ denote the longest element in $W$.  It is well known in Case $A_{n-1}$ that the length of $w_0$ is $\binom{n}{2}$.  We notate the reduced decomposition of any $w \in W$ by $s_{i_l} \cdots s_{i_1}$ for some $0 \leq l \leq \binom{n}{2}$ and $1
\leq i_1,...,i_l \leq n-1$, i.e. words are read from right to left.

\section{Weights and Highest Weight Modules}

Since $(\alpha_i, \alpha_i) = 2$ for every simple root, in Case $A_{n-1}$ we may describe the \textit{fundamental weights} $\omega_1, ... , \omega_{n-1}$ as being the elements of the dual basis to $\Delta$.  Let $1 \leq p \leq n-1$  Let $r = n - p$;  this number $r$ is the \textit{index reversal} of $p$.  For $1 \leq p \leq n-1$, we have $\omega_p = \frac{1}{n} (r,...,r,-p,...,-p)^T$, where there are $p$ $r$'s and $r$ $-p$'s.  Let $\Lambda  \subseteq \mathfrak{h}_{\mathbb{R}}^*$ denote the $\mathbb{Z}$-span of $\{ \omega_1, ... , \omega_{n-1} \}$.

A typical weight $\lambda \in \Lambda$ is written $\lambda = \sum_{i=1}^{n-1} a_i\omega_i$, where $a_i \in \mathbb{Z}$ for $1 \leq i \leq n-1$.  Recall that a non-zero weight $\lambda \in \Lambda$ is \textit{dominant} if $a_i \geq 0$ for all $i$; let $\Lambda^+$ denote the set of dominant weights.  Dominant weights correspond bijectively with the isomorphism classes of finite dimensional irreducible $\mathfrak{g}$-modules via the notion of  highest weight [Hu, p. 113].  For $\lambda \in \Lambda^+$, let $V(\lambda)$ be an irreducible module of highest weight $\lambda$.  Let $\Pi(\lambda)$ denote the set of weights of $V(\lambda)$ with respect to $\mathfrak{h}$.

\section{Formal Characters, Exponentials, and One Rowed Forms}

Let $\mathbb{Z}[\Lambda]$ denote the group ring of $\Lambda$ over $\mathbb{Z}$.  It is the free $\mathbb{Z}$-module over the elements of $\Lambda$.  For $\mu \in \Lambda$, let $e(\mu)$ denote the element in $\mathbb{Z}[\Lambda]$ associated to $\mu$.  For $\mu, \nu \in \Lambda$, define $e(\mu)e(\nu) := e(\mu + \nu)$.  Fix $\lambda \in \Lambda^+$.  For $\mu \in \Pi(\lambda)$, let $m_\lambda(\mu)$ denote the multiplicity of $\mu$ for $V(\lambda)$.  Define the \textit{formal character of $V(\lambda)$} as $\tilde{\chi}_\lambda := \sum_{\mu \in \Pi(\lambda)}
m_\lambda(\mu)e(\mu)$.

Let $\Lambda '$ denote the $\mathbb{Z}$-span of $\{ \epsilon_1, ... , \epsilon_n \}$ in $\mathbb{R}^n$.  For $1 \leq i \leq n$, let $t_i$ be a \textit{formal exponential} corresponding to $\epsilon_i \in \mathbb{R}^n \supseteq \mathfrak{h}_{\mathbb{R}}^*$.  The formal character may be concretely described by expressing each weight as a product of these formal exponentials $t_1, ... , t_n$.  Let $\mu \in \Lambda '$ and express $\mu = \mu_1 \epsilon_1 + ... + \mu_n \epsilon_n$.  Set $e(\mu) := t^{\mu} = t_1^{\mu_1} \cdots t_n^{\mu_n}$.  Denote the resulting character by $\tilde{\chi}_\lambda (t)$.  Note that $\tilde{\chi}_\lambda(t)$ is not a polynomial as the exponents are rational.

\begin{ex}
Let $n=3$ and $\lambda = \omega_1 + \omega_2 = \frac{1}{3}(2,-1,-1) + \frac{1}{3}(1,1,-2) = (1,0,-1)$.

Here $\Pi(\lambda) = \{ \lambda = \omega_1 + \omega_2, \hspace{.5mm} \omega_1, \hspace{.5mm} \omega_2, \hspace{.5mm} 2 \cdot 0 , -\omega_1, -\omega_2, -\omega_1 - \omega_2 \} = $

$\{ (1,0,-1), (0,1,-1), (1,-1,0), 2 \cdot (0,0,0), (0,-1,1), (-1,1,0), (-1, 0, 1) \}.$

So $\tilde{\chi}_\lambda(t) = t_1t_3^{-1} + t_2t_3^{-1} + t_1t_2^{-1} + 2 + t_2^{-1}t_3 + t_1^{-1}t_2+ t_1^{-1}t_3$.

\end{ex}

The induced action of the $s_i$'s on these formal exponentials is that of the transposition $(i, i+1)$.  For example, let $s_2s_1$ act on each member of the following list of exponentials: $(t_1, t_2, t_3)$.  (Note that this is not a column vector.)  Here we obtain $s_2s_1(t_1, t_2, t_3) = s_2 (t_2, t_1, t_3) = (t_3, t_1, t_2)$.  (Note that, in contrast, the action on a column vector in $\mathbb{R}^n$ is $s_2s_1 (z_1, z_2, z_3)^T = (z_2, z_3, z_1)$.)

Given a permutation $w$, the \textit{one rowed form} of $w$ is the result of letting $w$ act on each member of the list $(1, 2, ... , n)$ via these transpositions.  Denote the result of $w$'s one rowed form as $(x_1, ... , x_n)$; here $x_i$ is the image of $i$ under the action of $w$.  So the one rowed form of $s_2s_1$ is $(3, 1, 2)$.  This action is equivalent to the action $w(t_i) = t_{x_i}$ which appears in [RS, p. 108].  (To achieve the exact statement in [RS], replace the $t$'s above with $x$'s, and replace $x_i$ with $w_i$.)

\section{Relating Schur Functions to Formal Characters}

From the combinatorial description in Section 1.3, we see that $s_\beta(t)$ is a polynomial of homogeneous degree $N = |\beta|$.  One can rewrite this polynomial via $s_\beta(t) = (t_1 \cdots t_n)^{N/n} \tilde{s}_\beta(t)$ for some unique $\tilde{s}_\beta(t)$.  Note that $\tilde{s}_\beta(t)$ is homogeneous of degree zero.  Recall that  for $M \in SL_n(\mathbb{C})$ with eigenvalues $\theta_1, ... , \theta_n$, one has $\theta_1 \cdots \theta_n = 1$.  Thus when a $GL_n(\mathbb{C})$ character is restricted to $SL_n(\mathbb{C})$, the relation $t_1 \cdots t_n = 1$ is satisfied.  Combining these facts, we see that the restriction of $s_\beta(t)$ to $SL_n(\mathbb{C})$ yields $\tilde{s}_\beta(t)$.  Up to equivalence, the irreducible finite dimensional characters of $SL_n(\mathbb{C})$ are exactly the restrictions of the $GL_n(\mathbb{C})$ polynomial characters for the $n$-partitions $\beta$ with $\beta_n = 0$ [St, p. 445].

Let $\gamma$ be the sequence of nonnegative integers associated to $\beta$ from Section 1.2.  Let $\lambda = \sum_{i = 1}^{n-1} \gamma_i \omega_i \in \Lambda^+$.  Then it is known that $\tilde{s}_\beta(t) = \tilde{\chi}_\lambda (t)$.  Equivalently expressed as polynomials, we have $s_\beta(t) = (t_1 \cdots t_n)^{N/n} \tilde{\chi}_\lambda (t)$.  Using this, we can multiply the example $\tilde{\chi}_\lambda (t)$ from Section 1.9 by $(t_1t_2t_3)^{(3/3)} = t_1t_2t_3$ to obtain the example $s_\beta(t)$ from Section 1.2.  We denote the formal character expressed as a polynomial by $\chi_\lambda (t)$.

\section{Translated Fundamental Weights}

For $1 \leq p \leq n-1$, let $\eta_p = (1, ... , 1 , 0 , ... , 0)^T$ where there are $p$ 1's and $r$ 0's.  In case $A_{n-1}$, the vector $\eta_p$ is the translation of $\omega_p$ from $\mathfrak{h}_{\mathbb{R}}^*$ obtained by adding $p/n$ to each coordinate.  Since the vector $(p/n,...,p/n)^T$ is stabilized by all of $W$, it is harmless to use these translated versions $\eta_p$ for $\omega_p$.  Also, since the vector $(p/n,...,p/n)^T$ evaluates to zero on trace-free matrices, it is harmless to use these translations for the evaluation of fundamental weights.  If we express $\mu \in \Lambda$ as $\mu = \sum_{i=1}^{n-1} a_i\eta_i$ and set $e(\mu) = t^\mu$, then $\sum_{\mu \in \Pi(\lambda)} m_\lambda(\mu)e(\mu) = \chi_\lambda (t)$.  Henceforth we will use these translated versions of the fundamental weights so that all characters can be expressed as
polynomials.  Note that when we use the $\eta_p$'s, the elements of $\Lambda^+$ are in one-to-one correspondence with $n$-partitions $\beta$ for which $\beta_n = 0$.  From now on we will use $\lambda$ to denote an arbitrary $n$-partition (instead of $\beta$), and hence $N = |\lambda| = \lambda_1 + ... + \lambda_n$.

\section{Reflecting Weight Spaces}

Let $V$ be a finite dimensional irreducible $\mathfrak{g}$-module.  The module $V$ has at least one maximal vector of uniquely determined weight $\lambda$ for some $\lambda \in \Lambda^+$, and $V \cong V(\lambda)$ [Hu, p. 112].  For $\mu \in \Pi(\lambda)$, let $V_\mu$ denote the weight space of $V$ associated to $\mu$.  For $1 \leq i \leq n-1$, let $x_i, y_i$ denote the Chevalley generators of $\mathfrak{g}$.  Let $\phi : \mathfrak{g} \longrightarrow gl(V)$ denote the representation of $\mathfrak{g}$ afforded by $V$.  Define $s_i^\phi := (exp \hspace{3pt} \phi(x_i))(exp \hspace{3pt} \phi(-y_i))(exp \hspace{3pt} \phi(x_i))$.  Then one has $s_i^\phi (V_\mu) = V_{s_i(\mu)}$ [Hu, p. 119].  This operator $s_i^\phi$ allows us to move from one weight space to a reflected weight space.

\section{Demazure Modules}

Let $B = B(\Delta)$ denote the Borel subalgebra determined by $\Delta$.  Given a semisimple Lie algebra $L$, let $\mathfrak{U}(L)$ denote its universal enveloping algebra.  Fix $\lambda \in \Lambda^+$.  Let $v^+ \in V_\lambda$ denote a maximal vector (of weight $\lambda$), also known as a highest weight vector.  Fix $w \in W$.  Define the \textit{Demazure module of B} associated to $\lambda$ and $w$ to be $D(\lambda,w) := \mathfrak{U}(B).wv^+$.

\section{Demazure Characters}

Here we establish a character for $D(\lambda,w)$ similar to the formal characters defined earlier.  Let $\lambda \in \Lambda^+$ and $w \in W$.  Let $\Pi(\lambda, w)$ denote the set of weights of $D(\lambda,w)$ with respect to $\mathfrak{h}$.  Note that $\Pi(\lambda, w) \subseteq \Pi(\lambda)$.  For $\mu \in \Pi(\lambda, w)$, let $m_{\lambda,w}(\mu)$ denote the multiplicity of $\mu$ in $\Pi(\lambda, w)$.  Then the \textit{Demazure character} of $D(\lambda,w)$ is $\chi_{\lambda, w} := \sum_{\mu \in \Pi(\lambda, w)} m_{\lambda,w}(\mu) e(\mu)$.  By construction, we see that every term in $\chi_{\lambda, w}$ also appears in $\chi_\lambda$.  As for the formal characters of $sl_n(\mathbb{C})$, let $\chi_{\lambda, w}(t)$ denote the Demazure character when it is expressed as a polynomial in $t_1, ..., t_n$ in such a way that each monomial appearing in $\chi_{\lambda, w}(t)$ has degree $N$.

\section{The Demazure Character Formula}

The Demazure Character Formula presented in this section can be found in (for example) [Ku, p. 278]:  Given a simple reflection $s_i$ where $1 \leq i \leq n-1$, define the $\mathbb{Z}$-linear operator $D_{s_i}: \mathbb{Z}[\Lambda] \rightarrow \mathbb{Z}[\Lambda]$ by

$$D_{s_i}(e(\mu)) := \frac{e(\mu) - e(s_i(\mu)-\alpha_i)}{1-e(-\alpha_i)} $$

for $\mu \in \Lambda$.

Given $w \in W$, express $w = s_{i_l}...s_{i_1}$ in terms of simple reflections.  Define $D_w := D_{s_{i_l}} \circ ... \circ D_{s_{i_1}}$.

\begin{thm}
Fix $\lambda \in \Lambda^+$ and $w \in W$.  Then $D_w(e(\lambda)) = \chi_{\lambda, w}$.
\end{thm}

Note:  This Theorem applies to all Lie algebras types, though we will only apply it to type $A_{n-1}$.

Let $D_w(e(\lambda), t)$ denote $D_w(e(\lambda))$ expressed as a polynomial in $t_1, ... , t_n$.

\begin{ex}

Let n = 3, $\lambda = \eta_1 + \eta_2 = (2, 1, 0)$, and $w = s_2s_1$.

First we have $s_1(\lambda) - \alpha_1 = s_1(2, 1, 0) - (1, -1, 0) = (1, 2, 0) - (1, -1, 0) = (0, 3, 0)$.  So

$D_{s_1}(e(\lambda), t) = D_{s_1}(t_1^2t_2) = \frac{t_1^2t_2 - t_2^3}{1-t_1^{-1}t_2} = t_1^2t_2 + t_1t_2^2$.  Then

$D_{s_2}(t_1^2t_2) = \frac{t_1^2t_2 - t_1^2t_2^{-1}t_3^2}{1-t_2^{-1}t_3} = t_1^2t_2 + t_1^2t_3$, and

$D_{s_2}(t_1t_2^2) = \frac{t_1t_2^2 - t_1t_2^{-1}t_3^3}{1-t_2^{-1}t_3} = t_1t_2^2 + t_1t_2t_3 + t_1t_3^2$.  Therefore

$D_w(e(\lambda), t) = t_1^2t_2 + t_1^2t_3 + t_1t_2t_3 + t_1t_2^2 + t_1t_3^2$.

\end{ex}

\section{Key Polynomials}

The following definitions come from Chapter 2 of [RS].  For $1 \leq i \leq n-1$, define the linear operator $\delta_i: \mathbb{Z}[t_1,...,t_n] \rightarrow \mathbb{Z}[t_1,...,t_n]$ by

$$\delta_i t^\mu := \frac{t^\mu-s_i t^\mu}{t_i-t_{i+1}}$$

for $\mu \in \Lambda$.

Also define $\pi_i t^\mu:= \delta_i ( t_i t^\mu )$.  Given a permutation $w = s_{i_l}...s_{i_1}$, define $\pi_w := \pi_{i_l} \circ ... \circ \pi_{i_1}$.  Let $\theta = (\theta_1, ... , \theta_n)$ be a sequence of non-negative integers in $\mathbb{R}^n$.  Let $\lambda_\theta$ be the unique partition obtained by rearranging the entries of $\theta$.  Let $w_\theta$ be the shortest permutation that acts on the left by moving entries according to their locations (as opposed to according to their values) to obtain $\lambda_\theta = w_\theta \theta$.  Then define the \textit{key polynomial} $\kappa_\theta (t) := \pi_{w_\theta} t^{\lambda_\theta}$.

\begin{prop}
Let $\theta$ be a sequence of non-negative integers.  Then $\kappa_\theta (t) = D_{w_\theta} (e(\lambda_\theta), t )$.
\end{prop}

\begin{proof}
It suffices to show that if $\mu \in \Lambda$, then $D_{s_i}(e(\mu), t) = \pi_i (t^\mu)$.

$$D_{s_i}(e(\mu), t) = \frac{t^\mu - t^{s_i(\mu) - \alpha_i}}{1-t^{-\alpha_i}}
= \frac{t^\mu - t^{(\mu_1, \dots, \mu_{(i-1)}, (\mu_{(i+1)}-1),( \mu_i +1), \mu_{(i+2)} ,\dots, \mu_n)}}{1-t^{-\epsilon_i +\epsilon_{(i+1)}}}$$
$$= \frac{t_1^{\mu_1} \cdots t_n^{\mu_n} - t_1^{\mu_1} \cdots t_{i-1}^{\mu_{(i-1)}} (t_i^{\mu_{(i+1)}-1}) (t_{i+1}^{\mu_i + 1}) t_{i+2}^{\mu_{(i+2)}} \cdots t_n^{\mu_n}}{1-t_i^{-1}t_{i+1}}$$
$$=\frac{t_1^{\mu_1} \cdots t_{i-1}^{\mu_{(i-1)}} (t_i^{\mu_i+1}) t_{i+1}^{\mu_{(i + 1)}} \cdots t_n^{\mu_n} - t_1^{\mu_1} \cdots t_{i-1}^{\mu_{(i-1)}} (t_i^{\mu_{(i+1)}}) (t_{i+1}^{\mu_i + 1}) t_{i+2}^{\mu_{(i+2)}} \cdots t_n^{\mu_n} }{t_i - t_{i+1}} $$
$$=\frac{t_i t^\mu - s_i(t_1^{\mu_1} \cdots t_{i-1}^{\mu_{(i-1)}} (t_i^{\mu_i +1}) t_{i+1}^{\mu_{(i+1)}}) t_{i+2}^{\mu_{(i+2)}} \cdots t_n^{\mu_n})}{t_i - t_{i+1}}$$
$$=\frac{t_it^\mu - s_i(t_it^\mu)}{t_i-t_{i+1}} = \pi_i (t^\mu).$$

\end{proof}

It is a consequence of the definition of $\pi_i$ that if $a_i \geq a_{i+1}$, then  $\pi_i(t_1^{a_1} \cdots t_n^{a_n}) = t_1^{a_1} \cdots t_{i-1}^{a_{i-1}}t_{i+2}^{a_{i+2}} \cdots t_n^{a_n}(t_i^{a_i}t_{i+1}^{a_{i+1}} + t_i^{a_i - 1}t_{i+1}^{a_{i+1} + 1} + \dots + t_i^{a_{i+1}}t_{i+1}^{a_i})$.  (This is equivalent to Lemma 8.2.8 in [Ku, p. 279].)

\begin{ex}  Let $n=3$ and $\theta = (1, 0, 2)$.

Then $\lambda_\theta = (2, 1, 0)$, and $w_\theta = s_2s_1$.

Thus $\kappa_\theta (t) = \pi_2 \pi_1 (t_1^2t_2) = \pi_2 (t_1^2t_2 + t_1t_2^2) = t_1^2t_2 + t_1^2t_3 + t_1t_2^2 + t_1t_2t_3 +
t_1t_3^2$.

\end{ex}

\section{The Right Key}

A semistandard tableau $T \in \mathcal{T}_\lambda$ is a \textit{key} if the entries in the $j$th column of $T$ also appear in
the $(j-1)$st column for $1 < j \leq \lambda_1$.  Given two $n$-partitions $\lambda$ and $\beta$ such that $\beta_i \leq \lambda_i$ for $1 \leq i \leq n$, the \textit{skew diagram of shape $\lambda \backslash \beta$} is obtained by removing the Young diagram of $\beta$ from the Young diagram of $\lambda$.  Recall that $\zeta_\lambda$ is the sequence of column lengths of the Young diagram of shape $\lambda$.  Note that the length of the $j$th column in the skew diagram of shape $\lambda \backslash \beta$ is ${\zeta_\lambda}_j - {\zeta_\beta}_j$ for $1 \leq j \leq \lambda_1$.  A \textit{skew tableau} is a filling of a skew diagram with entries from $[n]$ such that the entries weakly increase along the rows and strictly increase down the columns.

Now we assume familiarity with the jeu de taquin (JDT) process,  as in (for example) Chapter 1 of [Fu].  The \textit{rectification} of a skew tableau [Fu, p. 15] is the tableau that results from successively applying ``JDT slides'' to the ``inside corners'' of the skew tableau (in any order) until there are no inside corners left.  A skew tableau is \textit{frank} if its column lengths are a rearrangement of the column lengths of its rectification.  Given $T \in \mathcal{T}_\lambda$, the \textit{right key of $T$}, denoted $R(T)$, is a key of shape $\lambda$ defined
column-wise as follows:  For $1 \leq j \leq \lambda_1$, the $j$th column of $R(T)$ is equal to the rightmost column of any frank skew tableau that
rectifies to $T$ and whose rightmost column has length ${\zeta_\lambda}_j$.  Given $T \in \mathcal{T}_\lambda$, its right key $R(T)$ is well-defined [Fu, p. 209].  For further details, including an example and two specific ways to compute right keys, see Sections 2, 3, and 4 of Chapter 2.

\section{A Tableau Description for Key Polynomials and Demazure Characters}

Given a sequence of non-negative integers $\theta = (\theta_1, ... , \theta_n)$, the \textit{key of $\theta$}, denoted $Y_{\lambda_\theta}(\theta)$, is the key of shape $\lambda_\theta$ whose first $\theta_j$ columns contain $j$ for $1 \leq j \leq {\lambda_\theta}_1$.  The following result is Theorem 1 of [RS]:

\begin{prop}
$\kappa_\theta (t) = \sum_{ \{ T \in \mathcal{T}_{\lambda_\theta} | R(T) \leq Y_{\lambda_\theta}(\theta) \} } wt(T)$.
\end{prop}

Given an $n$-partition $\lambda$ and a permutation $w$ expressed in one rowed notation $(x_1, ... , x_n)$, the \textit{key of $w$ (of shape $\lambda$)}, denoted $Y_\lambda(w)$, is the key whose columns of length $k$ consist of $x_1, x_2, ... , x_k$ arranged in increasing order, for all $1 \leq k \leq n$.  It can be seen that $Y_{\lambda_\theta}(\theta) = Y_{\lambda_\theta}(w_\theta)$.

Given $\lambda \in \Lambda^+$ and $w \in W$, a \textit{Demazure tableau of shape $\lambda$ for $w$} is a semistandard tableau $T$ of shape $\lambda$ such that $R(T) \leq Y_\lambda(w)$.  We set $\mathcal{D}_\lambda(w) := \{ T \in \mathcal{T}_\lambda \hspace{1mm} | \hspace{1mm} R(T) \leq Y_\lambda(w) \}$.  Using Proposition 1.16.1, we can restate Proposition 1.18.1:

\begin{thm}
$\chi_{\lambda, w} (t) = \sum_{T \in \mathcal{D}_\lambda(w)} wt(T)$.
\end{thm}

The set $\mathcal{D}_\lambda(w)$ is called the \textit{set of Demazure tableaux (for $\lambda$ and $w$)}.

\begin{ex}
Let $n = 3$, $\lambda = \eta_1 + \eta_2 = (2,1,0)$, and $w = s_2s_1$.  The one rowed form of $w$ is $(3, 1, 2)$.

$$\mathcal{T}_\lambda = \{ \hspace{1mm}
\ytableausetup{centertableaux, boxsize = 1.25em}
\begin{ytableau}
1 & 1 \\
2  \\
\end{ytableau}, \hspace{1mm}
\ytableausetup{centertableaux, boxsize = 1.25em}
\begin{ytableau}
1 & 1 \\
3  \\
\end{ytableau}, \hspace{1mm}
\ytableausetup{centertableaux, boxsize = 1.25em}
\begin{ytableau}
1 & 2 \\
2  \\
\end{ytableau}, \hspace{1mm}
\ytableausetup{centertableaux, boxsize = 1.25em}
\begin{ytableau}
1 & 2 \\
3  \\
\end{ytableau}, \hspace{1mm}
\ytableausetup{centertableaux, boxsize = 1.25em}
\begin{ytableau}
1 & 3 \\
2  \\
\end{ytableau}, \hspace{1mm}
\ytableausetup{centertableaux, boxsize = 1.25em}
\begin{ytableau}
1 & 3 \\
3  \\
\end{ytableau}, \hspace{1mm}
\ytableausetup{centertableaux, boxsize = 1.25em}
\begin{ytableau}
2 & 2 \\
3  \\
\end{ytableau}, \hspace{1mm}
\ytableausetup{centertableaux, boxsize = 1.25em}
\begin{ytableau}
2 & 3 \\
3  \\
\end{ytableau} \hspace{1mm}
\}.$$

The corresponding right keys of these tableaux are
$$\{ \hspace{1mm}
\ytableausetup{centertableaux, boxsize = 1.25em}
\begin{ytableau}
1 & 1 \\
2  \\
\end{ytableau}, \hspace{1mm}
\ytableausetup{centertableaux, boxsize = 1.25em}
\begin{ytableau}
1 & 1 \\
3  \\
\end{ytableau}, \hspace{1mm}
\ytableausetup{centertableaux, boxsize = 1.25em}
\begin{ytableau}
1 & 2 \\
2  \\
\end{ytableau}, \hspace{1mm}
\ytableausetup{centertableaux, boxsize = 1.25em}
\begin{ytableau}
2 & 2 \\
3  \\
\end{ytableau}, \hspace{1mm}
\ytableausetup{centertableaux, boxsize = 1.25em}
\begin{ytableau}
1 & 3 \\
3  \\
\end{ytableau}, \hspace{1mm}
\ytableausetup{centertableaux, boxsize = 1.25em}
\begin{ytableau}
1 & 3 \\
3  \\
\end{ytableau}, \hspace{1mm}
\ytableausetup{centertableaux, boxsize = 1.25em}
\begin{ytableau}
2 & 2 \\
3  \\
\end{ytableau}, \hspace{1mm}
\ytableausetup{centertableaux, boxsize = 1.25em}
\begin{ytableau}
2 & 3 \\
3  \\
\end{ytableau} \hspace{1mm}
\}.$$

Also, $Y_\lambda(w) = \ytableausetup{centertableaux, boxsize = 1.25em}
\begin{ytableau}
1 & 3 \\
3  \\
\end{ytableau}$.

Hence

$$\mathcal{D}_\lambda(w) = \{ \hspace{1mm}
\ytableausetup{centertableaux, boxsize = 1.25em}
\begin{ytableau}
1 & 1 \\
2  \\
\end{ytableau}, \hspace{1mm}
\ytableausetup{centertableaux, boxsize = 1.25em}
\begin{ytableau}
1 & 1 \\
3  \\
\end{ytableau}, \hspace{1mm}
\ytableausetup{centertableaux, boxsize = 1.25em}
\begin{ytableau}
1 & 2 \\
2  \\
\end{ytableau}, \hspace{1mm}
\ytableausetup{centertableaux, boxsize = 1.25em}
\begin{ytableau}
1 & 3 \\
2  \\
\end{ytableau}, \hspace{1mm}
\ytableausetup{centertableaux, boxsize = 1.25em}
\begin{ytableau}
1 & 3 \\
3  \\
\end{ytableau} \hspace{1mm}
\},$$

and $\chi_{\lambda, w}(t) = t_1^2t_2 + t_1^2t_3 + t_1t_2^2 + t_1t_2t_3 + t_1t_3^2$.

\end{ex}

\chapter{A Direct Way to Find the Right Key of a Semistandard Young Tableau}

\section{Introduction}

A key is a semistandard tableau such that the entries in each column also appear in each column to the left.  If a semistandard tableau $T$ is not a
key itself, then the right (``left'') key of $T$ is a slightly greater (lesser) tableau (via entrywise comparison) that is associated to $T$.

Perhaps the foremost result that uses the notion of the right key of a tableau is the summation formula of Section 1.18 for Demazure characters due
to Lascoux and Sch{\"u}tzenberger [LS].  Recall the well known description of Schur functions from Section 1.2:  Given a partition $\lambda$, the Schur function $s_\lambda$ is equal to the sum of the weights of all semistandard tableaux of shape $\lambda$.  More generally, given a highest weight $\mu$ and a permutation $w$, the formula of Lascoux and Sch{\"u}tzenberger says that the Demazure character (or key polynomial) determined by $\mu$ and $w$ of a ``Borel subalgebra'' of $sl_n(\mathbb{C})$ is equal to the sum of the weights of all semistandard tableaux of shape $\mu$ whose right keys are less than or equal to the key corresponding to $w$.

In this chapter, given a semistandard tableau $T \in \mathcal{T}_\lambda$, we introduce a ``scanning method'' which produces a key of
shape $\lambda$ associated to $T$.  We provide an English description of the scanning method in Section 2.2, and then a precise mathematical description in Section 2.3.  In Section 2.5, we prove that this key equals the right key of $T$.  Then in Section 2.6, we describe an analogous method for producing the left key of $T$.  Lastly, in Section 2.7, we will discuss some previous methods for right and left key computations.

\section{An English Description of the Scanning Method}

We now describe a method to quickly write down the right key of a given semistandard tableau $T$.  To do so we need only one special definition beyond
the usual terminology.  Given a sequence $(x_1,x_2,x_3,...)$, its \textit{earliest weakly increasing subsequence} (EWIS) is
$(x_{i_1},x_{i_2},x_{i_3},...)$, where $i_1 = 1$ and for $j>1$ the index $i_j$ is the smallest index such that $x_{i_j} \geq x_{i_{j-1}}$.

Here is an English description of our method that a non-mathematician could apply:  Draw an empty Young diagram that has the same shape as $T$; this
diagram will receive the entries of the right key of $T$.  Viewing the bottom entries of the columns of $T$ from left to right as forming a sequence,
use two index fingers to find the EWIS of this sequence.  Whenever an entry is added to the EWIS, put a dot above it.  When this EWIS ends, write its
last member in the diagram for the right key of $T$ in the lowest available box in the leftmost available column.  Next, repeat the process as if the
boxes with dots are no longer part of $T$.  Once every box in $T$ has a dot, the leftmost column of the right key has been formed.  To find the
entries of the next column in the right key: ignore the leftmost column of $T$, erase the dots in the remaining boxes, and repeat the above process.
Continue in this manner until the Young diagram has been filled with the entries of the right key of $T$.

It can be seen that if there is more than one column of a given length, upon reaching this length one may ignore all but the rightmost of these
columns.  Also, erasing the dots is unnecessary if there is one color of ink available for each distinct column length.

\begin{ex}
Let $T$ be the semistandard tableau shown in Figure 2.1.  First we find the EWISs that begin in the first column of $T$.  To display all six EWISs with
just one figure, we use the over-decorations $\dot{\textcolor[rgb]{1.00,1.00,1.00}{1}}$, $\bar{\textcolor[rgb]{1.00,1.00,1.00}{1}}$,
$\tilde{\textcolor[rgb]{1.00,1.00,1.00}{1}}$, $\acute{\textcolor[rgb]{1.00,1.00,1.00}{1}}$, $\hat{\textcolor[rgb]{1.00,1.00,1.00}{1}}$ and
$\check{\textcolor[rgb]{1.00,1.00,1.00}{1}}$ to indicate their successive creations and removals.
\end{ex}

\begin{figure}[h]
\caption{}

\center
\ytableausetup{centertableaux, boxsize = 2em}
\begin{ytableau}
\check{1} & \check{1} & \hat{3} & \hat{4} & \acute{6} \\
\hat{2} & \hat{3} & \acute{5} & \tilde{7} & \dot{9} \\
\acute{4} & \tilde{5} & \tilde{6} & \bar{8} \\
\tilde{5} & \bar{7} & \dot{9} \\
\bar{7} \\
\dot{8}
\end{ytableau}
\hspace{25mm}
\begin{ytableau}
1 & 6 & 6 & 6 & 6 \\
4 & 7 & 7 & 7 & 9 \\
6 & 8 & 8 & 9 \\
7 & 9 & 9 \\
8 \\
9
\end{ytableau}

\end{figure}

So the six EWISs are (8,9,9), (7,7,8), (5,5,6,7), (4,5,6), (2,3,3,4), and (1,1).  Hence the entries of the 1st column of the right key of $T$ are 9,
8, 7, 6, 4, and 1.  To find the EWISs for the later columns, we repeat the process (after erasing the over-decorations) as if the first column is no
longer part of $T$.  The four EWISs for the rightmost length 4 column are (9,9), (6,8), (5,7), and (3,4,6).  Continuing the procedure for the length 3 column, the three EWISs are (8,9), (7), and (4,6).  For the length 2 column (and in fact the rightmost column of any SSYT), the EWISs are just the
entries in the column.  Writing in the last members from each of these EWISs, we obtain the right key of T, which is also displayed in Figure 1.

This chapter's main result, Theorem 2.4.5, states that the output of this method is in fact the right key of $T$.  The proof uses the jeu de taquin
description of the right key presented in [Fu].

\section{Scanning Tableau}

Now we describe the above procedure with mathematical notation.  A \textit{shape} $\zeta$ is a sequence $\zeta = (\zeta_1 , \zeta_2, ... , \zeta_k)$
of positive integers such that $\zeta_1 \geq \zeta_2 \geq ... \geq \zeta_k$ for some $k \geq 0$.  Each member of a shape is a \textit{(column)
length}. The \textit{Young diagram of $\zeta$} is the diagram with $\zeta_i$ top justified boxes in column $i$ for $1 \leq i \leq k$.  A
\textit{semistandard tableau of shape $\zeta$} is a filling of the Young diagram of $\zeta$ with positive integers such that the entries are strictly
increasing down the columns and weakly increasing across the rows.  Henceforth we more simply write \textit{tableau (of shape $\zeta$}).  (For Lie
theoretical purposes, one would bound both the lengths of a shape and the tableau entries by some fixed $n \geq 1$.)

Fix a tableau $T$ of shape $\zeta$ with $k$ columns.  Let $T_i$ denote the $i$th column of $T$ for $1 \leq i \leq k$.  The tableau $T$ is a
\textit{key} if the entries in $T_i$ also appear in $T_{i-1}$ for $1 < i \leq k$.  Given a column $C$, let $l(C)$ denote its length.  So $l(T_i) =
\zeta_i$.  Define $b_i$ to be the bottom entry of $T_i$ for $1 \leq i \leq k$.  Let $e = (e_1,...,e_m)$ for some $m \leq k$ denote the EWIS obtained
from $b = (b_1,...,b_k)$.  Let $d = \{ d_1, ... , d_m \}$ denote the collection of boxes in the Young diagram of $T$ that respectively contain the
entries of $e$.  Define $\setminus$ to be the operator that removes a specified set of boxes (and their entries) from a diagram or tableau.  It is
easy to see that $T \setminus d$ is a tableau.

Write $x \bigodot C$ to notate the result of attaching a single entry $x$ to the bottom of a column $C$, and $C \bigoplus T$ to notate the result of
prepending a column $C$ to the left side of $T$.  The following doubly recursive definition specifies a tableau $S(T)$ that has the same shape as $T$;
here $S_1(T)$ denotes the first column of $S(T)$ and $S_{1B}(T)$ denotes the bottom entry of $S_1(T)$.  Set $S_{1B}(T):= e_m$ (if $T$ is nonempty;
otherwise, for terminating purposes set $S_{1B}(T):= \emptyset$).  Set $S_1(T):= S_{1B}(T) \bigodot S_1(T \setminus d)$.  Define $S(T):= S_1(T)
\bigoplus S(T \setminus T_1)$.  We call $S(T)$ the \textit{scanning tableau} of $T$.  The fact that $S(T)$ is a key will follow from the main result
of this paper.

\section{The Right Key}

In this section we present a common definition of the right key of a tableau, as can be found in Appendix 5 of [Fu].  We also present a specific method
for calculating the right key that uses this definition.  Given a shape $\zeta$ with $k$ lengths and a shape $\eta$ with $l \leq k$ lengths such that
$\eta_i \leq \zeta_i$ for $1 \leq i \leq l$, the \textit{skew diagram of shape $\zeta \setminus \eta$} is obtained by removing the diagram of $\eta$
from the diagram of $\zeta$.  The sequence of column lengths of $\zeta \setminus \eta$ is $(\zeta_1 - \eta_1, ... , \zeta_l - \eta_l, \zeta_{l+1}, ...
, \zeta_k)$.  Some of the lengths in a skew diagram may be zero.  A \textit{skew tableau} is a filling of a skew diagram with positive integers such
that the entries are strictly increasing down the columns and weakly increasing along the rows.

Now we assume familiarity with the jeu de taquin (JDT) process,  as in (for example) Chapter 1 of [Fu].  The \textit{rectification} of a skew tableau
[Fu, p. 15] is the tableau that results from successively applying ``JDT slides'' to the ``inside corners'' of the skew tableau (in any order) until
there are no inside corners left.  It is known that the rectification of a skew tableau is unaffected by any JDT slide or ``reverse JDT (RJDT)
slide''.

\begin{lem}
 Let $U$ be the skew tableau of some shape $\zeta \setminus \eta$ that is formed by arranging the $k$ columns $U_1, ... , U_k$ according to $\eta$.
 Let $1 \leq l \leq k$ and $d > 0$.  Using RJDT slides, one may shift all of the entries of each of $U_1 , ... , U_l$ down $d$ rows without otherwise
 modifying $U$.
\end{lem}

\begin{proof}
As $j$ runs from 1 to $l$, successively perform $d$ RJDT slides beginning directly under $U_j$.  It can be seen that for $2 \leq j \leq l$, each of
these RJDT slides will pull each entry of $U_j$ down one row (without affecting the columns that have already been pulled down).
\end{proof}

A skew tableau is \textit{frank} if its lengths are a rearrangement of the lengths of the shape of its rectification.

\begin{prop}[]
\textbf{[Fu, p. 208]}  Given a tableau $T$ and a skew diagram $\zeta \setminus \eta$ whose lengths are a rearrangement of the lengths of $T$, there is
a unique skew tableau on $\zeta \setminus \eta$ that rectifies to $T$.
\end{prop}

\begin{cor}
The rightmost column of a frank skew tableau that rectifies to $T$ depends only upon the length of that column.
\end{cor}

Fix a shape $\zeta$ with $k$ lengths.  Let $T$ be a tableau of shape $\zeta$.  For $1 \leq i \leq k$, define $R_i(T)$ to be the rightmost column of
any frank skew tableau whose rightmost column has length $\zeta_i$ and that rectifies to $T$.  The \textit{right key of} $T$ is $R(T) :=
\bigoplus_{i=1}^k R_i(T)$.  Note that $R(T)$ has shape $\zeta$.  Also, it is known [Fu, App. A.5 Cor. 2] that $R(T)$ is a key.  Before specifying a
method to obtain the individual columns of $R(T)$, we define an operation that will be iterated to produce each of them.

Let $U$ be a skew tableau of some shape $\zeta \setminus \eta$ formed from the $k$ columns $U_1 , ... , U_k$, where $\eta$ has $l \leq k$ lengths.
For any $l < j < k$, set $x_j := l(U_j) - l(U_{j+1})$ and define $d_j$ to be the number of boxes in $U_{j-1}$ that are attached to boxes in $U_j$.
For such $j$, the process of successively shifting $U_1, ... , U_{j-1}$ down $d_j$ rows and then successively performing $x_j$ RJDT slides on the
outside corner of $U_j \bigoplus U_{j+1}$ is called the \textit{jth length swap (on $U$)}.  Note that the lengths of the resulting skew tableau are
$(l(U_1), ... , l(U_{j-1}), l(U_{j+1}), l(U_j), l(U_{j+2}), ... , l(U_k))$.

\begin{lem}
Let $U$ be a frank skew tableau of shape $\zeta \setminus \eta$ that rectifies to a tableau $T$.  For $l < j < k$, the $j$th length swap produces a
frank skew tableau $U'$ that also rectifies to $T$.
\end{lem}

\begin{proof}
Rectification is preserved by RJDT slides.  Since $U$ is frank and the lengths of $U'$ are a rearrangement of the lengths of $U$, we have that $U'$ is also frank.
\end{proof}

\begin{prop}
Let $T$ be a tableau of shape $\zeta$ with $k$ columns and let $1 \leq i < k$.  Successively performing the $i$th, $(i+1)$st, ... , $(k-1)$st length
swaps produces a frank skew tableau with lengths $(\zeta_1, ... , \zeta_{i-1} , \zeta_{i+1} , ... , \zeta_k, \zeta_i)$  that rectifies to $T$.  Its
rightmost column is $R_i(T)$.
\end{prop}

\begin{proof}
Apply Lemma 2.4.4 as $j$ runs from $i$ to $k-1$.
\end{proof}

\section{Main Result Concerning the Right Key}

The method of producing $R(T)$ in this section parallels the construction of $S(T)$ in Section 2.3.  Let $\zeta$ be a shape with $k$ lengths.  Fix a
tableau $T$ of shape $\zeta$.

\begin{lem}
$R(T) = R_1(T) \bigoplus R(T \setminus T_1)$.
\end{lem}

\begin{proof}
Proposition 2.4.5 implies that the calculation of $R_i(T)$ does not involve the first $i-1$ columns of $T$.  Therefore the calculations for
$\bigoplus_{i=2}^k R_i(T)$ are independent of $T_1$.
\end{proof}

We now focus on producing $R_1(T)$.  Define $T^{(1)} := T$ and for $1 \leq j \leq k-1$ define $T^{(j+1)}$ to be the skew tableau obtained by
performing the $j$th length swap on $T^{(j)}$.  The shape of $T^{(j)}$ has lengths $(\zeta_2, ... , \zeta_j, \zeta_1, \zeta_{j+1}, ... , \zeta_k)$.
Therefore $T_k^{(k)} = R_1(T)$.  Let $R_{1B}(T)$ denote the bottom entry of $R_1(T)$.  Recall that $b_i$ is the bottom entry of $T_i$ for $1 \leq i
\leq k$, and that $b = (b_1, ... , b_k)$.  For $1 \leq j, h \leq k$, let $b_h^{(j)}$ denote the bottom entry of $T_h^{(j)}$.  Since $R_1(T) =
T_k^{(k)}$, we have that $R_{1B}(T) = b_k^{(k)}$.  Note that if $h > j$, then $T_h^{(j)} = T_h$ and hence $b_h^{(j)} = b_h$.  For $1 \leq j \leq k-1$,
the entry $b_{j+1}^{(j+1)}$ produced by the $j$th length swap on $T^{(j)}$ is one of two possible entries:

\begin{lem}
Let $1 \leq j \leq k-1$. \\
(i)  If $b_{j+1}^{(j)} \geq b_j^{(j)}$, then $b_{j+1}^{(j+1)} = b_{j+1}^{(j)}$. \\
(ii)  Otherwise, $b_{j+1}^{(j+1)} = b_j^{(j)}$.
\end{lem}

\begin{proof}
Let $1 \leq j \leq k-1$.  The difference in lengths of the two rightmost columns involved in the $j$th length swap on $T^{(j)}$ is $x_j =  \zeta_1 -
\zeta_{j+1}$.  First suppose $b_{j+1}^{(j)} \geq b_j^{(j)}$.  In this case, by the semistandard conditions $b_{j+1}^{(j)}$ is greater than every entry
in $T_{j}^{(j)} \bigoplus T_{j+1}^{(j)}$ except perhaps $b_j^{(j)}$.  Thus the first $x_j-1$ RJDT slides of the $j$th length swap will each pull
$b_{j+1}^{(j)}$ down one row.  Then the $x_j$th RJDT slide will compare $b_j^{(j)}$ to $b_{j+1}^{(j)}$ and also pull $b_{j+1}^{(j)}$ down. Hence
$b_{j+1}^{(j+1)} = b_{j+1}^{(j)}$.  Otherwise, $b_{j+1}^{(j)} < b_j^{(j)}$.  In this case, $b_j^{(j)}$ is greater than every entry in $T_{j}^{(j)}
\bigoplus T_{j+1}^{(j)}$.  As a result, the $x_j$th RJDT slide will move $b_j^{(j)}$ to the right regardless of what it is compared to.  Thus
$b_{j+1}^{(j+1)} = b_j^{(j)}$.
\end{proof}

Note that {\sl (ii)} happens precisely when the final RJDT slide of the $j$th length swap moves the bottom entry of $T_j^{(j)}$ across to become the
bottom entry of $T_{j+1}^{(j+1)}$.  In this case, the EWIS $e$ of $b$ ``skips'' $T_{j+1}$.  The complementary possibility {\sl (i)} occurs precisely
when a new member is appended to $e$.  Hence:

\begin{cor}
$R_{1B}(T) = S_{1B}(T)$.
\end{cor}

\begin{proof}
Apply these two observations as $j$ run from 1 to $k-1$ in Lemma 2.5.2 to see that $b_k^{(k)}$ is the last member of $e$.
\end{proof}

Recall that $d$ denotes the collection of boxes containing $e$:

\begin{lem}
$R_1(T) = R_{1B}(T) \bigodot R_1(T \setminus d)$.
\end{lem}

\begin{proof}
Let $1 \leq j \leq k-1$.  We need to show that the entries of $T_{j+1}^{(j+1)}$ above $b_{j+1}^{(j+1)}$ determined by the $j$th length swap on
$T^{(j)}$ are unaffected by the presence of the entries of $e$ in $T_j^{(j)}$.  (These entries are initially located at $d$ in $T^{(1)} = T$.) It is
clear that:  $(*)$ For $1 \leq j \leq k-1$, the presence of $b_j^{(j)}$ will not affect whether or not the entries in $T_j^{(j)}$ above it will be
pulled to the right during the $j$th length swap.  Since $e_1 = b_1^{(1)}$, it can be safely removed from $T$ (and its successor skew tableaux).

From Lemma 5.2, all but the first member of $e$ arise as $b_{j+1}^{(j)}$ for some $1 \leq j \leq k-1$.  Suppose $1 \leq j \leq k - 1$ and
$b_{j+1}^{(j)} = b_{j+1}^{(j+1)}$.  By the proof of Lemma 2.5.2 we have that $b_{j+1}^{(j)}$ was pulled straight down from the bottom of $T_{j+1}^{(j)}$
to the bottom of $T_{j+1}^{(j+1)}$.  So the first move in every RJDT slide executed in the $j$th length swap pulls $b_{j+1}^{(j)}$ down 1 row.  Then
the remainder of the RJDT slide continues as it would have if $b_{j+1}^{(j)}$ was not originally part of $T^{(j)}$.  Now $b_{j+1}^{(j)} =
b_{j+1}^{(j+1)}$, so by $(*)$ it can be safely removed from $T$ (and its successor skew tableaux).  Similar reasoning shows that $b_k^{(k)}$ can be
safely removed when $b_k^{(k-1)} = b_k^{(k)}$.  So we see that the presence of the entries in $e$ does not affect the determination of $R_1(T)
\setminus R_{1B}(T)$.  Thus once $R_{1B}(T)$ is found, we can remove all of $d$ (and $e$) from $T$ and then find the rest of $R_1(T)$ iteratively
without the entries of $e$ being present.
\end{proof}

\begin{thm}
Let $T$ be a tableau.  The right key $R(T)$ is equal to the scanning tableau $S(T)$.
\end{thm}

\begin{proof}
The combination of Corollary 2.5.3, Lemma 2.5.4, and Lemma 2.5.1 agrees with the doubly recursive definition of $S(T)$.
\end{proof}

\section{The Left Key}

Here is a method to write down the left key of a given semistandard tableau $T$:  Draw an empty Young diagram that has the same shape as $T$.  Let $k$
be the number of columns of $T$.  Let $a = (a_1, ... , a_k)$ be the sequence defined as follows:  Let $a_1$ be the bottom entry of the $k$th column.
For $1 < i \leq k$, let $a_i$ be the largest entry in the $(k+1-i)$th column that is less than or equal to $a_{i-1}$.  Note that by the semistandard
row condition, at least one such entry is guaranteed to exist.  Put a dot above each entry of $a$.  Place $a_k$ in the lowest available box in the
rightmost available column of the Young diagram.  Next, repeat the process as if the boxes with dots and all boxes below them are no longer part of
$T$.  Once such a sequence has been found for each entry in the rightmost column of $T$, the rightmost column of the left key of $T$ has been
determined.  To find the next column in the left key, ignore the rightmost column of $T$, erase all the dots, and repeat the above process.  Continue
in this manner until the entire Young diagram has been filled.

The proof that the tableau produced by this method is the left key of $T$ is similar to the proof of the analogous statement concerning the right key
given in this paper.  In fact the proof is simpler, since the sequence produced above necessarily contains exactly one entry from each column to the
left of its starting point.

\section{Previous Methods}

Lascoux and Sch\"{u}tzenberger introduced [LS] the notions of right and left key in Definition 2.9 and developed their relation to Demazure characters on pp. 132 - 138.  The paper [RS] by Reiner and Shimozono explicitly stated the summation formula for the ``key polynomial'' in Theorem 1.  Fulton's
book [Fu] appeared in 1997.  In [Le] Lenart stated (on p. 280) that one can use ``reverse column insertion'' to produce the left key of a tableau,
restates [RS, Thm. 1] as Theorem 4.1, and discusses how this result may be proved using material from [LS].  In [Av], Aval showed how to compute the left key of a tableau using manipulations of ``sign matrices''.  Then he obtained the right key from the left key by using the notion of ``complement'' of a Young tableau.  Mason presented a method in [Mas] to produce the right key of a tableau using ``semiskyline augmented fillings''.  Both [Av] and [Mas] had larger goals than key computations that led to these procedures.

\chapter{Sets of Demazure Tableaux and Their Convexity}

\section{Introduction}

In this chapter, we will begin in Section 3.2 by introducing some notation to keep track of the scanning paths used in constructing $S(T)$ for a given semistandard tableau $T$.  Then in Section 3.3, we will fix a partition $\lambda$ and a permutation $w$ and study the structure of a Demazure tableaux $T$.  Given a location $(l,k)$ in $T$, we will define a ``local condition set'' for $(l,k)$ based on the key of $w$ and the entries ``southwest'' of $(l,k)$ in $T$.  The first main result in this chapter, Theorem 3.4.1, shows that a tableau $T$ is a Demazure tableau for $w$ if and only if every entry of $T$ belongs to the local condition set for its location.  Thus we obtain an entry-wise condition for a semistandard tableau to contribute to the Demazure character for $\lambda$ and $w$.

Starting in Section 3.6, we define and study a pattern, the ``312-pattern'', that may arise in the one rowed permutation of a given $w$.  Reiner and Shimozono showed in Theorem 25 of [RS] that a certain tableau condition on $w$ is necessary and sufficient for a Demazure character to equal a ``flagged Schur polynomial''.  These flagged Schur polynomials are generalizations of the Schur functions described in Section 1.2 for which the tableaux entries satisfy certain bounds based on their rows.  Postnikov and Stanley improved upon one direction of Theorem 25 in Theorem 14.1 of [PS]:  They showed that if $w$ is ``312-avoiding'', then the Demazure character is also a flagged Schur polynomial.

For our strengthenings of these results, due to time limitations we restrict to the $n$-partition ``staircase shapes'' $\lambda = (d, ... , d, d-1, ... , 2, 1)$ for some $1 \leq d \leq n$.  (The results of [PS] and [RS] consider general $\lambda$.)  In Proposition 3.7.1, we prove that if $w$ is 312-avoiding, then the Demazure tableaux for $w$ are precisely the semistandard tableaux that are less than or equal to the key of $w$.  This result pertains to the actual set of tableaux that contribute to the Demazure character, as introduced by Lascoux and Sch{\"u}tzenberger in [LS], whereas the results of [PS] and [RS] are stated only at the character level.  We are able to do this using our Theorem 3.4.1; no direct description of general Demazure tableaux was available to earlier authors.

In Chapter 15 of [PS], Postnikov and Stanley noted that the set of tableaux associated to a flagged Schur polynomial is a ``convex polytope'' in $\mathbb{R}^{l(w)}$.  It is a quick consequence of Proposition 3.7.1 that the set of  Demazure tableaux for a 312-avoiding permutation is a convex polytope in $\mathbb{R}^{|\lambda|}$, as noted in Corollary 3.7.2.  In a forthcoming paper we will indicate how one can see that the set of Demazure tableaux for a 312-avoiding permutation is the set of tableaux for a certain flagged Schur polynomial.  In Section 3.8, we prove that if $\lambda = (n, n-1, ... , 2, 1)$ and $w$ is not 312-avoiding, then the set of Demazure tableaux is not a convex polytope.  Since this pattern avoiding condition is both necessary and sufficient, the final result in this chapter, Theorem 3.9.1, combines these two results to completely characterize for which $w \in W$ the set of Demazure tableaux for $\lambda = (n, n-1, ... , 2, 1)$ is a convex polytope.

\section{Scanning Path Notation and the Local Condition Set}

Fix an $n$-partition $\lambda \in \Lambda^+$ and a permutation $w \in W$.  Fix a semistandard tableau $T \in \mathcal{T}_\lambda$.  Here we recall the scanning method of Section 2.3 for constructing $S(T)$ and introduce some additional notation.  Let $1 \leq j \leq \lambda_1$.  Working from the bottom of the $j$th column of $T$ upwards, we construct one ``northeasterly'' \textit{scanning path} in the shape $\lambda$ originating from each box in that column of $T$.  Set $\lambda^{(1)} := \lambda$ (hence $\zeta^{(1)} = \zeta$) and $T^{(1)} := T$.  Initialize $P(T; j, \zeta_j) := ( (j, \zeta^{(1)}_j) )$.  Scan the column bottoms $T^{(1)}(h, \zeta^{(1)}_h)$ for $h > j$ for the earliest $h$ such that $T^{(1)}(j, \zeta^{(1)}_j) \leq T^{(1)}(h, \zeta^{(1)}_h)$.  If such an $h$ exists, append $(h, \zeta^{(1)}_h)$ to $P(T; j, \zeta_j)$.  Repeat this process for the bottom entries further to the right of $(h, \zeta^{(1)}_h)$, comparing them to the entry in the most recently appended location until there does not exist a further weakly larger entry.  This completes the construction of the scanning path $P(T; j, \zeta_j)$.

The entries of $T$ used in this process form the earliest weakly increasing subsequence (EWIS) for the sequence of column bottoms weakly to the right of $(j, \zeta_j)$.  Define $S(T; j, \zeta_j)$ to be the last value in this EWIS.  Now remove the boxes in $P(T; j, \zeta_j)$ from $\lambda^{(1)}$ to create $\lambda^{(2)}$ (and hence $\zeta^{(2)}$), and the boxes and their values from $T^{(1)}$ to obtain $T^{(2)}$.  As $i$ runs from 2 to $\zeta_j$, repeatedly apply this process to $T^{(i)}$ to produce the other $\zeta_j - 1$ scanning paths that begin in the $j$th column.  Here each constructed path $P(T^{(i)}; j, \zeta^{(i)}_j)$ is henceforth denoted $P(T; j, \zeta_j - (i-1))$.  Apply this process to all columns of $T$ to obtain $S(T; j,i)$ for every $(j,i) \in \lambda$.  These are the entries of the scanning tableau $S(T)$ of $T$.

Fix $(l, k) \in \lambda$.  For each $j \leq l$, it can be seen that there is exactly one $i \in [1, \zeta_j]$ such that $(l,k) \in P(T; j,i)$.  Now fix some $1 \leq j \leq l-1$.  Let $a(l,k; j) =: a(j)$ be the row index such that $(l-1, k) \in P(T; j, a(j))$.  If $k < \zeta_l$, let $b(l,k; j) =: b(j)$ be the row index such that $(l, k+1) \in P(T; j, b(j))$.  When $k = \zeta_l$, set $b(l,k; j) = \zeta_j + 1$.  It can be seen that the only paths beginning in column $j$ that may reach $(l,k)$ are the paths originating from rows $a(j)$ through row $b(j) - 1$ inclusive.  For $a(j) \leq i \leq b(j)-1$, let $h$ be the largest value less than $l$ such that $(h, m) \in P(T; j, i)$ for some m.  Then define $E(l, k; j, i) := T(h,m)$, where $h$ and $m$ depend upon $l,k,j,i$ as above.  By convention, set $a(l) := b(l)-1 := k$ and $E(l,k;l,k) :=k$.  For any $(j,i) \in \lambda$, the entry at $(j,i)$ in $Y_\lambda(w)$ is denoted $Y(w;j,i) := Y_\lambda(w;j,i)$.  Define the \textit{local condition set} $B(T, w; l, k) := \bigcap_{j=1}^{l} \hspace{2mm}( \bigcup_{i=a(j)}^{b(j)-1} \hspace{2mm} [ E(l, k; j,i), Y(w; j,i)  ] \hspace{2mm} )$.

\section{Using Local Condition Sets for Right Key Comparisons}

Let $\lambda \in \Lambda^+$ and $w \in W$.  Here we show that the scanning tableau of a given semistandard tableau $T$ is dominated by the key of $w$ if and only if the entries in $T$ come from the local condition sets for their locations.

\begin{prop}
Let $T \in \mathcal{T}_\lambda$.  Then $S(T; j, i) \leq Y(w; j, i)$ for all $(j,i) \in \lambda$ if and only if $T(l, k) \in B(T, w; l, k)$ for all $(l,k) \in \lambda$.
\end{prop}

\begin{proof}

First fix $(l, k) \in \lambda$.  Let $1 \leq j \leq l-1$.  Let $1 \leq i \leq \zeta_j$ be the unique index such that $(l,k) \in P(T; j,i)$.  The last value before $T(l,k)$ in the EWIS defining $P(T; j,i)$ was denoted $E(l,k;j,i)$.  The last value in this EWIS is $S(T; j, i)$.  So we have $E(l,k; j,i) \leq T(l,k) \leq S(T;j,i)$.  By hypothesis we have $S(T;j,i) \leq Y(w;j,i)$.  Hence $T(l,k) \in [ E(l,k; j,i), Y(w; j,i) ]$.  When $j=l$, we have $\bigcup_{i=a(j)}^{b(j)-1}  [ E(l, k; j,i), Y(w; j,i)  ]  ) = [k, Y(w; l,k)]$.  Since $T$ is semistandard, we know $T(l,k) \geq k$.  From the definition of $S(T;l,k)$, we have $T(l,k) \leq S(T; l,k)$.  Hence $T(l,k) \in [k, Y(w; l,k)]$.  Intersecting over all $1 \leq j \leq l$, we see that $T(l,k) \in B(T,w;l,k)$.

Now fix $(j, i) \in \lambda$.  Let $(l,k)$ denote the last position in $P(T; j,i)$; here $S(T; j,i) = T(l,k)$.  By hypothesis, we have $T(l,k) \in \bigcup_{h = a(j)}^{b(j)-1} [ E(l,k; j,h), Y(w; j,h) ]$.  However, the entry $T(l,k) < E(l,k; j,h)$ for all $h > i$.  (Otherwise $(l,k)$ would be in $P(T; j,h)$ for some $h > i$.)  So $T(l,k) \in \bigcup_{h = a(j)}^i [ E(l,k; j,h) , Y(w; j,h) ]$.  Since $Y_\lambda(w)$ is semistandard, we have $Y(w; j, r) > Y(w, j, s)$ when $r > s$.  Thus $Y(w; j, i)$ is an upperbound for $\bigcup_{h = a(j)}^i [ E(l,k; j,h), Y(w; j,h) ]$.  This implies $S(T; j,i) = T(l,k) \leq Y(w; j, i)$.

\end{proof}

\section{A New Description of Demazure Tableaux and Characters}

Fix $\lambda \in \Lambda^+$ and  $w \in W$.  Recall that $\mathcal{D}_\lambda(w) = \{ T \in \mathcal{T}_\lambda \hspace{1mm} | \hspace{1mm} R(T) \leq Y_\lambda(w) \}$ denotes the set of Demazure tableaux for $w$.  Theorem 2.5.5 allows us to avoid refering to the right key $R(T)$ in the definition of $\mathcal{D}_\lambda(w)$, since it is equal to the scanning tableau $S(T)$.  Hence Proposition 3.3.1 may be used here to reformulate the notion of a Demazure tableau.  The following result is joint work with R. A. Proctor:

\begin{thm}
Let $\lambda \in \Lambda^+$ and $w \in W$.  A semistandard tableau $T$ of shape $\lambda$ is a Demazure tableau for $w$ if and only if $T(l, k) \in B(T, w; l, k)$ for all $(l, k) \in \lambda$.
\end{thm}

Now recall Theorem 1.18.2, which says that $\chi_{\lambda, w}(t) = \sum_{T \in \mathcal{D}_\lambda(w)} wt(T)$.

\begin{cor}  The Demazure character $\chi_{\lambda, w} (t)$ is the sum of $wt(T)$ over all semistandard tableaux $T$ such that $T(l,k) \in B(T, w; l, k)$ for all $(l,k) \in \lambda$.
\end{cor}

So an assignment of values from $[n]$ to the boxes in $\lambda$ is a Demazure tableau for $w$ if and only if $T(l,k) \in [T(l-1,k),T(l,k+1)-1] \cap B(T,w;l,k)$.  The set $\mathcal{T}_\lambda$ of all semistandard tableaux of shape $\lambda$ may be recursively constructed ``from the southwest'' via a backtracking procedure by requiring that each successively proposed additional entry $T(l, k)$ be selected from $[T(l-1,k) , T(l, k+1)-1]$.  From Theorem 3.4.1, since the definition of $B(T,w;l,k)$ refers only to entries ``lying to the southwest'', it can be seen that the set $\mathcal{D}_\lambda(w)$ of all Demazure tableaux for $w$ of shape $\lambda$ may be similarly constructed:  Now require $T(l,k) \in [T(l-1,k), T(l, k+1)-1] \cap B(T,w;l,k)$.  (However, for the actual computation of $\mathcal{D}_\lambda(w)$, one should also consider the generation procedure described on p. 281 of [Le].)

\section{Convex Polytopes}

\vspace{.5pc}\hspace{4mm}  Fix $\lambda \in \Lambda^+$ and $w \in W$.  In this section we consider the set $\mathcal{D}_\lambda(w)$ of Demazure tableaux as a subset of integral Euclidean space.  Recall that $N = \lambda_1 + ... + \lambda_n$.  Consider each $T \in \mathcal{D}_\lambda(w)$ as a point in $\mathbb{Z}^N$ by reading its entries in some fixed order.  A subset $\mathcal{S}$ of $\mathbb{Z}^N$ is a \textit{convex polytope} if for some $m \geq 1$ there exists an $m$ x $N$ real matrix $A$ and an $m$ x $1$ column vector $b$ such that $\mathcal{S}$ is the set of integral solutions to the inequality $Ax \leq b$.  From this definition it can be seen that if a convex polytope contains the endpoints of a line segment in $\mathbb{R}^N$, then it must also contain all of the integral points on that line segment between those endpoints.

\section{A Lemma for 312-Avoiding Permutations}

Fix $w \in W$ and express it in one rowed form as $(x_1, ... , x_n)$.  The permutation $w$ is \textit{312-avoiding} if there does not exist $1\leq a < b < c \leq n$ such that $x_b < x_c < x_a$; otherwise $w$ is \textit{312-containing}.  We say a column in a tableau has a \textit{gap} if two consecutive entries differ by more than one.  A key $K$ of shape $\lambda$ has a \textit{312-containing gap} if there exists an index $1 \leq j \leq \lambda_1-1$ such that the $j$th column has a gap somewhere below the entry that appears in the $j$th column but does not appear in the $(j+1)$st column.

\begin{lem}
Let $\lambda$ be the $n$-partition $(d, ... , d, d-1, ... , 2, 1)$ for some $1 \leq d \leq n$, and let $w \in W$.  If $w$ is 312-avoiding then $Y_\lambda(w)$ has no 312-containing gaps.
\end{lem}

\begin{proof}
We will prove the contrapositive.  First, let $1 \leq j \leq d-1$ be an index for which $Y_\lambda(w)$ has a 312-containing gap, and let $b = n+1-j$.  Here $x_b$ is the entry that appears in the $j$th column but not the $(j+1)$st column of $Y_\lambda(w)$.  Each value omitted by the gap is of the form $x_c > x_b$ for some $c \in [b+1,n]$.  Choose any such $x_c$.  Every entry in the $j$th column below the gap below $x_b$ is of the form $x_a > x_c$ for some $a \in [1, b-1]$.  Choose any such $x_a$.  Then $w$'s one rowed form contains $(..., x_a , ... , x_b , ... , x_c , ...)$ with $x_b < x_c < x_a$, i.e. $w$ is 312-containing.

\end{proof}

\section{A Sufficient Condition for Convexity}

In this section we prove that sets of Demazure tableaux for 312-avoiding $w$'s do not require right key computations.  Moreover, the simpler description implies that $\mathcal{D}_\lambda(w)$ is a convex polytope.  The following result holds for all $1 \leq d \leq n$; this includes the most important ``staircase shape'' case when $d = n$, which yields $\lambda = (n, n-1, ... , 2, 1)$.

\begin{prop}
Let $\lambda$ be the $n$-partition $(d, ... , d, d-1, ... , 2, 1)$ for some $1 \leq d \leq n$.  If $w$ is 312-avoiding, then $\mathcal{D}_\lambda(w) = \{ T \in \mathcal{T}_\lambda \hspace{1mm} | \hspace{1mm} T \leq Y_\lambda(w) \}$.
\end{prop}

Since all $T$ in $\mathcal{D}_\lambda(w)$ here are semistandard and merely less than or equal to $Y_\lambda(w)$, this proposition tells us that the local condition set $B(T, w; l,k)$ has no effect on the possible values for $T(l,k)$ for any $(l,k) \in \lambda$.  Set $Y(w;l,k) := Y_\lambda(w;l,k)$.

\begin{proof}

Let $\lambda$ be an $n$-partition and let $w \in W$.  Theorem 3.4.1 says that $\mathcal{D}_\lambda(w) = \{ T \in \mathcal{T}_\lambda \hspace{1mm} | \hspace{1mm} T(l,k) \in B(T, w;l,k)$ for all $(l,k) \in \lambda \}$.  Let $T \in \mathcal{D}_\lambda(w)$ and $(l,k) \in \lambda$.  Taking $j = l$ in the intersection defining $B(T,w;l,k)$, we see that $T(l,k) \in [k, Y(w; l,k)]$.  So $T(l,k) \leq Y(w;l,k)$ (for any $\lambda$ and any $w$).  Hence $\mathcal{D}_\lambda(w) \subseteq \{ T \in \mathcal{T}_\lambda \hspace{1mm} | \hspace{1mm} T \leq Y_\lambda(w) \}$.

To show the other containment $\{ T \in \mathcal{T}_\lambda \hspace{1mm} | \hspace{1mm} T \leq Y_\lambda(w) \} \subseteq \mathcal{D}_\lambda(w)$, we must now assume that $\lambda$ is of the assumed form with $d$ columns and that $w$ is 312-avoiding.  For this containment we use induction on $d$.  Let $[\lambda]_d$ denote the result of truncating $\lambda$ down to its first $d$ columns and let $[T]_d$ denote the result of truncating $T$ down to its first $d$ columns.

Suppose $d = 1$.  Here $\lambda = (1, ... , 1)$, so $T \in \mathcal{T}_\lambda$ is a semistandard column of length $n$.  Clearly $S(T) = T$ for all $T \in \mathcal{T}_\lambda$ here.  Hence $\{ T \in \mathcal{T}_\lambda \hspace{1mm} | \hspace{1mm}  T \leq Y_\lambda(w) \} = \{ T \in \mathcal{T}_\lambda  \hspace{1mm} |  \hspace{1mm} S(T) \leq Y_\lambda(w) \} = \mathcal{D}_\lambda(w)$.

Now suppose $1 < d \leq n$, and so $\lambda = (d, ... , d, d-1, ... , 2, 1)$ has $d$ columns.  Assume $\{ T \in \mathcal{T}_\mu \hspace{1mm} | \hspace{1mm}  T \leq Y_\mu(w) \} =  \mathcal{D}_\mu(w)$ when $\mu$ is of this form with fewer than $d$ columns.  Let $T \in \mathcal{T}_\lambda$ be such that $T \leq Y_\lambda(w)$, and let $(l,k) \in \lambda$.  We know that $T(l,k) \in [T(l-1,k), T(l,k+1)]$, and that $T(l,k) \in [T(l-1,k), Y(w;l,k)]$.  So it suffices to fix a column and to show that the corresponding union within the intersection $B(T,w;l,k)$ contains $[T(l-1,k),T(l,k+1)-1]$ or $[T(l-1,k),Y(w;l,k)]$.

Note that when $l=1$, the set $B(T,w;1,k)$ is exactly the latter interval.  So assume $2 \leq l \leq d = \lambda_1$.  Fix $1 \leq j < l$.  Recall that $E(l,k;j,a(j)) = T(l-1,k)$.  Note that Case 2 below is the complement of the union of Cases 1 and 3.

Case 1:  Suppose $Y(w; j, a(j)) \geq Y(w; l,k)$.  In this case $\bigcup_{i=a(j)}^{b(j)-1} [ E(l,k;j,i) , Y(w; j,i)]$ contains $[E(l,k;j,a(j)) , Y(w; j,a(j))] = [T(l-1,k), Y(w; j,a(j))]$, which in turn contains $[T(l-1,k) , Y(w; l,k)]$ here.

Case 2:  Suppose $Y(w; j, a(j)) < Y(w; l,k)$ and $Y(w; j, b(j)-1) \geq Y(w; l,k)$.  Since $a(j) \geq k$, the first assumption in this case implies that there exists some $m \leq a(j)$ such that $Y(w; j, m) =:p$ does not appear in the $l$th column of $Y_\lambda(w)$.  Let $r$ be maximal such that $p$ does not appear in the $r$th column of $Y_\lambda(w)$; so $j < r \leq l$.  Let $q$ be the number of boxes below $p$ in the $(r-1)$st column.  Lemma 3.6.1 implies that the lowest $(q+1)$ boxes in this column contain the consecutive entries $p, p+1, p+2, ... , p+q$.  Since $m \leq a(j)$, we have $p \leq Y(w;j,a(j))$.  Since $Y_\lambda(w)$ is a key, the entry $Y(w;l,k)$ also appears in the $(r-1)$st column, hence $Y(w;l,k) \leq p+q$.  Then by the first assumption, we see that  $Y(w;j,a(j)), Y(w;l,k) \in [p, p+q]$.

Since $Y_\lambda(w)$ is a key, all of $[p, p+q]$ also appears in the $j$th column.  Let $h = Y(w;l,k) - Y(w;j,a(j))$.  From this we see that $Y(w;j,a(j)+i) = Y(w;j,a(j)) + i$ for $1 \leq i \leq h$.  By the second assumption, $Y(w;l,k)$ appears weakly above $Y(w;j,b(j)-1)$ in the $j$th column.  Thus $\bigcup_{i=a(j)}^{b(j)-1} [ E(l,k;j,i) , Y(w; j,i)]$ contains the intervals $[E(l,k;j,a(j)),Y(w;j,a(j))]$,  $[E(l,k;j,a(j)+1); Y(w;j,a(j)) + 1]$, $ ... $ , $[E(l,k;j,a(j)+h),Y(w;l,k)]$.

Note that the condition $T \leq Y_\lambda(w)$ implies that $[T]_{l-1} \leq Y_{[\lambda]_{l-1}}(w)$.  Since $1 \leq l-1 \leq d-1$, the induction hypothesis gives $[T]_{l-1} \in \mathcal{D}_{[\lambda]_{l-1}}$.  Using the scanning characterization of a Demazure tableau, it can be seen that any value in the path $P(T;j,i)$ is less than or equal to $Y(w;j,i)$.  Since $E(l,k;j,i) \in [T]_{l-1}$ when $a(j) \leq i \leq b(j)-1$, we have $E(l,k;j,i) \leq Y(w;j,i)$.  Hence $[E(l,k;j,i),Y(w;j,i)]$ contains at least $Y(w;j,i)$ for $a(j) \leq i \leq b(j)-1$.  So the union above contains $[T(l-1,k), Y(w;l,k)]$.

Case 3:  Suppose $Y(w; j, b(j)-1) < Y(w; l,k)$.  Since $Y_\lambda(w)$ is a key, the entry $Y(w;l,k)$ appears in the $j$th column.  So the assumption here implies $b(j) \leq \zeta_j$.  Thus in this case the location $(l, k+1)$ must be in $\lambda$.  Since $Y(w;j,a(j)) \leq Y(w;j,b(j)-1)$, the assumption here implies, as in Case 2, that $p, p+1, p+2, ... , p+q$ all appear consecutively in the $j$th column, where $p$ and $q$ are defined as in Case 2.  Again we have $Y(w;j,a(j)), Y(w;l,k) \in [p, p+q]$.  Consequently we see that $Y(w;j,b(j)-1) \in [p, p+q-1]$.  Thus $Y(w;j,b(j)) \in [p+1,p+q]$.

Now we have that $Y(w;j,a(j)+i) = Y(w;j,a(j))+i$ for $1 \leq i \leq b(j)-a(j)$.  So $\bigcup_{i=a(j)}^{b(j)-1} [ E(l,k;j,i) , Y(w; j,i)] $ contains the intervals $[E(l,k;j,a(j)), Y(w;j,a(j))]$,  $[E(l,k;j,a(j)+1),Y(w;j,a(j))+1]$ , ... , $[E(l,k;j,b(j)-1),Y(w;j,b(j))-1] $.  As in Case 2, their union contains $[T(l-1,k), Y(w;j,b(j))-1]$.  Here we have that $(l,k+1) \in P(T;j,b(j))$, and so $T(l,k+1) \leq Y(w; j,b(j))$.  Hence $[T(l-1,k), Y(w;j,b(j)) - 1]$ contains $[T(l-1,k), T(l,k+1) - 1]$.

So in all three cases, we see that $\bigcup_{i=a(j)}^{b(j)-1} [ E(l,k;j,i) , Y(w; j,i)]$ contains $[T(l-1,k), T(l,k+1) - 1]$ or $[T(l-1,k), Y(w;l,k)]$.  When $j=l$, the inclusion of $[T(l-1,k), Y(w;l,k)]$ is immediate.  Thus $T(l,k)$ is in the intersection $B(T, w;l, k)$ for all $(l,k) \in \lambda$.  Hence $\{ T \in \mathcal{T}_\lambda \hspace{1mm} | \hspace{1mm} T \leq Y_\lambda(w) \} \subseteq \mathcal{D}_\lambda(w)$ when $\lambda$ is of the form with $d$ columns.

\end{proof}

\vspace{.5pc}  This Proposition quickly implies:

\begin{cor}
Let $\lambda$ be the $n$-partition $(d, ... , d, d-1, ... , 2, 1)$ for some $1 \leq d \leq n$.  If $w$ is 312-avoiding, then $\mathcal{D}_\lambda(w)$ is a convex polytope.
\end{cor}

\begin{proof}
\vspace{.5pc}  The inequalities $T(l,k) \geq T(l-1,k)$, $T(l,k) \leq T(l,k+1) - 1$, and $T(l,k) \leq Y(w;l,k)$ for all $(l,k) \in \lambda$ can be encoded into the system of inequalities $AT \leq b$ for some $m$ x $N$ matrix A and an $m$ x 1 vector $b$.

\vspace{.5pc}  Thus $\{ T \in T_\lambda \hspace{1mm} | \hspace{1mm} T \leq Y_\lambda(w) \} = \{ T \hspace{1mm} | \hspace{1mm} AT \leq b \}$.

\end{proof}

\section{A Necessary Condition for Convexity}

In this section we will show that if $\lambda = (n, n-1, ... , 2, 1)$ and $w$ is 312-containing, then the set of Demazure tableaux $\mathcal{D}_\lambda(w)$ does not contain a particular semistandard tableau that lies on the line segment defined by two other particular tableaux that are in $\mathcal{D}_\lambda(w)$.  Thus we obtain a necessary condition for $\mathcal{D}_\lambda(w)$ to be a convex polytope, namely that it be 312-avoiding.

\begin{prop}  Let $\lambda = (n, n-1, ... , 2, 1)$.  If $w$ is 312-containing, then $\mathcal{D}_\lambda(w)$ is not a convex polytope.
\end{prop}

\begin{proof}

Since $w = (x_1, ... , x_n)$ is 312-containing, there exists some $a < b < c$ such that $x_b < x_c < x_a$.  Amongst such patterns, we specify one that is optimal for our purposes:  Choose $b$ to be minimal.  Then choose $c$ to be minimal with respect to this $b$.  Then choose $a$ so that $x_a$ is minimal with respect to this $x_c$.

These minimal choices lead to the following three conditions:  (i) By the minimality of $b$, there does not exist $a < z < b$ such that $x_b < x_z < x_c$.  (ii) By the minimality of $c$, there does not exist $b < z < c$ such that $x_b < x_z < x_a$.  (iii) By the minimality of $x_a$, there does not exist $z < b$ such that $x_c < x_z < x_a$.  If there exists $e < a$ such that $x_b < x_e < x_c$, choose $d < a$ such that $x_d$ is maximal with respect to this condition; otherwise set $d = b$.  (iv)  By the maximality of $x_d$, there does not exist $z < a$ such that $x_d < x_z < x_c$.

Now let $w' = (x_1' , ... , x_n')$ be the result of swapping $x_b$ with $x_d$ in the one rowed form of $w$; so $x_b' = x_d, x_d' = x_b$, and $x_i' = x_i$ when $i \neq b, d$.  (If $d = b$, then $w' = w$ with $x_b' = x_b$ and $x_d' = x_d$.)  Let $j = n + 1 - b$; so the entry $x_b'$ appears precisely in the 1st through $j$th columns of $Y_\lambda(w')$.  Let $k = b - d$.  The swap producing $w'$ from $w$ replaces the $x_d$'s in the $(j+1)$st through $(j+k)$th columns of $Y_\lambda(w)$ with $x_d' = x_b$ to produce $Y_\lambda(w')$.  (Here the entries in these columns may need to be re-sorted to meet the semistandard criteria.)  So $x_d' \leq x_d$ implies $Y_\lambda(w') \leq Y_\lambda(w)$ via a column-wise argument.

Note that by (i), (iii), and (iv), in the one rowed form of $w'$ there does not exist $z < b$ such that $x_d = x_b' < x_z < x_a$.  Thus we obtain:  (v) The $(j+1)$st through $n$th columns of $Y_\lambda(w')$ do not contain any entries from $[x_b', x_a - 1]$.  Let $(j,i)$ denote the location of the $x_b'$ in the $j$th column of $Y_\lambda(w')$ (and hence $Y_\lambda(w)$).  By (i), (iii), and (iv) we can see that $Y(w';j, i+1) = x_a$.  Since $x_b'$ is the only entry from the $j$th column of $Y_\lambda(w')$ that does not appear in the $(j+1)$st column, we also have that $Y(w'; j+1, i) = x_a$.

Let $v$ be the permutation obtained by swapping $x_b' = x_d$ with $x_a' = x_a$ in $w'$.  Let $m = b - a$.  (Note that $n + 1 - a = j + m$.)  Let $L \subseteq \lambda$ be the set of locations of the $x_a$'s in the $(j+1)$st through $(j+m)$th columns of $Y_\lambda(w')$.  By (v), obtaining $v$ from $w'$ is equivalent to literally replacing the $x_a$'s at every location in $L$ in $Y_\lambda(w')$ with $x_b'$ (and leaving the rest of $Y_\lambda(w')$ unchanged) to obtain $Y_\lambda(v)$.  So $x_b' < x_a$ implies $Y_\lambda(v) < Y_\lambda(w')$.

Let $T(c)$ be the result of literally replacing the $x_a$'s at every location in $L$ in $Y_\lambda(w')$ with $x_c$ (and leaving the rest unchanged).  Note that $Y(w';p,q) = Y(v;p,q) = T(c;p,q)$ for every $(p,q) \notin L$.  Since $Y_\lambda(w')$ and $Y_\lambda(v)$ are keys and hence semistandard, for every location $(p,q) \in L$ we have $Y(w';p,q), Y(v;p,q) \in [Y(w';p,q-1) + 1, Y(w';p,q+1) - 1]$ and $Y(w';p,q), Y(v;p,q) \in [Y(w';p-1,q) , Y(w';p+1,q)]$.  Then since $Y(v;p,q) = x_b' < x_c < x_a = Y(w';p,q)$ for every $(p,q) \in L$, we also have $T(c;p,q) \in [Y(w';p,q-1) + 1, Y(w';p,q+1)-1]$ and $T(c;p,q) \in [Y(w';p-1,q), Y(w';p+1,q)]$ for every $(p,q) \in L$.  The entry at every location in $T(c)$ that is not in $L$ is precisely the entry for that location in $Y_\lambda(w')$.  There is at most one location from $L$ in any column of $Y_\lambda(w')$.  Since there will be no failures of the semistandard conditions within $L$ after placing the same value at each location in $L$, we can conclude that $T(c)$ is semistandard, i.e. we have $T(c) \in \mathcal{T}_\lambda$.

In $\mathbb{R}^N$, consider the line segment $u(s) = Y_\lambda(v) + s(Y_\lambda(w')-Y_\lambda(v))$, where $0 \leq s \leq 1$.  Note that $u(0) = Y_\lambda(v)$ and $u(1) = Y_\lambda(w')$.  Here the value of $s$ only affects the entries at the locations in $L$.  Let $r = \frac{x_c - x_b'}{x_a - x_b'}$; note that $0 < r < 1$.  These entries in $L$ in $u(r)$ are $x_b' + \frac{x_c - x_b'}{x_a-x_b'}(x_a-x_b') = x_c$.  Hence $u(r) = T(c)$.

Consider $S(T(c)):$  Since $T(c;j,i+1) = x_a > x_c = T(c;j+1,i)$, the location $(j+1,i) \notin P(T(c);j,i+h)$ for any $h \geq 1$.  Since $T(c;j,i) = x_b' < x_c$, the location $(j+1,i) \in P(T(c);j,i)$.  Thus $S(T(c);j,i) \geq x_c > x_b' = Y(w;j,i)$.  Hence $S(T(c)) \nleq Y_\lambda(w)$, so $T(c) \notin \mathcal{D}_\lambda(w)$.

Since $Y_\lambda(w')$ and $Y_\lambda(v)$ are keys we have $R(Y_\lambda(w')) = Y_\lambda(w')$ and $R(Y_\lambda(v)) = Y_\lambda(v)$.  Then $Y_\lambda(v) < Y_\lambda(w') \leq Y_\lambda(w)$ implies $Y_\lambda(w')$, $Y_\lambda(v) \in \mathcal{D}_\lambda(w)$.  Thus $u(0), u(1) \in \mathcal{D}_\lambda(w)$ but $u(r) \notin \mathcal{D}_\lambda(w)$.  Therefore $\mathcal{D}_\lambda(w)$ is not a convex polytope.

\end{proof}

\section{Main Result for Convexity with a Discussion of Earlier Work}

In this section we present the main result of Chapter 3, and then compare and contrast its two constituent results to the results of [PS] and [RS].  Combining the foremost case of Corollary 3.7.2 with Proposition 3.8.1, we obtain:

\begin{thm}  Let $\lambda = (n, n-1, ... , 2, 1)$.  Fix $w \in W$.  Then $\mathcal{D}_\lambda(w)$ is a convex polytope if and only if $w$ is 312-avoiding.
\end{thm}

In Theorem 25 of [RS], Reiner and Shimozono present a sufficient condition on $w$ expressed with a tableau for a Demazure character to equal a flagged Schur polynomial.  The hypothesis of Proposition 3.7.1 translates that condition into the more common 312-avoiding condition for a permutation.  In addition, Proposition 3.7.1 shows that when $w$ is 312-avoiding, the set of actual tableaux used to produce the Demazure character is exactly the set of tableaux which will be shown in a forthcoming paper to produce the corresponding flagged Schur polynomial.  Theorem 25 of [RS] is a statement at only the character level.  (However, Theorem 25 of [RS] holds for general $n$-partitions $\lambda$, whereas Proposition 3.7.1 is limited to the ``partial staircase shapes'' $\lambda = (d, ... , d, d-1, ... , 2, 1)$ for $1 \leq d \leq n$.)  We hope to extend Proposition 3.7.1 to general shapes in a future paper.

Theorem 25 of [RS] also presents a necessary condition for a Demazure character to equal a flagged Schur polynomial.  The contrapositive of Proposition 3.8.1 translates that condition into the 312-avoiding condition for a permutation.  Proposition 3.8.1 implies that the set of Demazure tableaux for a 312-containing $w$ is not equal to any set of flagged Schur tableaux.  (Again, Theorem 25 of [RS] holds for general $\lambda$, while Proposition 3.8.1 holds only for $\lambda = (n, n-1, ... , 2, 1)$.)

After formulating Propositions 3.7.1 and 3.8.1, we learned of the paper [PS] by Postnikov and Stanley.  Their Theorem 14.1 also improves upon the statement of Theorem 25 of [RS] by substituting the 312-avoiding hypothesis to imply that the Demazure character for $\lambda$ and $w$ is a flagged Schur polynomial.  In Section 15 of [PS], Postnikov and Stanley note that the set of tableaux contributing to a flagged Schur polynomial is a convex polytope in $\mathbb{R}^{l(w)}$, as was observed by Kogan [Ko].  However, since their Theorem 14.1 was presented at the character level, Postnikov and Stanley were unable to view the set of Demazure tableaux themselves as a convex polytope in $\mathbb{R}^d$ for some $d$.  Our Theorem 3.9.1 completely characterizes for which $w \in W$ the set of Demazure tableaux for $\lambda = (n, n-1, ... , 2, 1)$ is a convex polytope in $\mathbb{R}^{|\lambda|}$.

\end{document}